\documentclass[12pt]{amsart}
\usepackage{mathtools}
\usepackage{amssymb}
\usepackage{mathrsfs}
\usepackage[pagebackref,colorlinks=true,linkcolor=blue,citecolor=blue]{hyperref}
\usepackage{tikz-cd}

\usepackage{stmaryrd}

\UseRawInputEncoding

\tikzset{
  symbol/.style={
    draw=none,
    every to/.append style={
      edge node={node [sloped, allow upside down, auto=false]{$#1$}}}
  }
}

\newcommand{\Z}{\mathbb{Z}}

\newcommand{\R}{\mathbb{R}}
\newcommand{\BC}{\mathbb{C}}

\newcommand{\OO}{{\mathcal O}}

\newcommand{\ol}{\overline}

\newcommand{\SL}{\mathrm{SL}}
\newcommand{\GL}{\mathrm{GL}}
\newcommand{\SO}{\mathrm{SO}}

\newcommand{\Sp}{\mathrm{Sp}}

\newcommand{\Fr}{\mathrm{Fr}}
\newcommand{\RG}{\mathrm{G}}

\newcommand{\WF}{\overline{\mathfrak{n}}^{m}}

\newcommand{\half}[1]{\frac{#1}{2}}

\newcommand{\comment}[1]{}

\newtheorem{thm}{Theorem}[section]
\newtheorem{cor}[thm]{Corollary}
\newtheorem{lemma}[thm]{Lemma}
\newtheorem{prop}[thm]{Proposition}
\newtheorem {conj}[thm]{Conjecture}

\newtheorem {ques/conj}[thm]{Question/Conjecture}

\newtheorem{defn}[thm]{Definition}
\newtheorem{remark}[thm]{Remark}

\newtheorem{exmp}[thm]{Example}

\newtheorem*{globalcond*}{Global Condition}
\newtheorem*{localcond*}{Local Condition}

\newtheorem*{globalconj*}{Global Conjecture}
\newtheorem*{localconj*}{Local Conjecture}

\newtheorem*{nonzero*}{Conjecture on the non-vanishing of the normalized intertwining operators}
\newtheorem*{holo*}{Conjecture on the holomorphicity of the normalized intertwining operators}

\DeclareMathOperator{\Ad}{Ad}

\numberwithin{equation}{section}

\let\oldbullet\bullet
\renewcommand{\bullet}{{\vcenter{\hbox{\tiny$\oldbullet$}}}}

\begin{document}
\title[Weak local Arthur packets conjecture]
{On the weak local Arthur packets conjecture for split classical groups}

\author[Baiying Liu]{Baiying Liu}
\address{Department of Mathematics\\
Purdue University\\
West Lafayette, IN, 47907, USA}
\email{liu2053@purdue.edu}

\author[Chi-Heng Lo]{Chi-Heng Lo}
\address{Department of Mathematics\\
Purdue University\\
West Lafayette, IN, 47907, USA}
\email{lo93@purdue.edu}

\subjclass[2000]{Primary 11F70, 22E50; Secondary 11F85, 22E55}

\date{\today}


\keywords{Admissible Representations, Local Arthur Packets, Weak Local Arthur Packets, Local Arthur Parameters, Nilpotent Orbits, Wavefront Sets}

\thanks{The research of the first named author is partially supported by the NSF Grants DMS-1702218, DMS-1848058}

\maketitle

\begin{abstract}
Recently, motivated by the theory of real local Arthur packets, making use of the wavefront sets of representations over non-Archimedean local fields $F$, 
Ciubotaru, Mason-Brown, and Okada defined the weak local Arthur packets consisting of certain unipotent representations and conjectured that they are unions of local Arthur packets. In this paper, we prove this conjecture for split classical groups with the assumption of the residue field characteristic of $F$ being large. 
In particular, this implies the unitarity of these unipotent representations.  
We also discuss the generalization of the weak local Arthur packets beyond unipotent representations which reveals the implications of a conjecture of Jiang on the structure of wavefront sets for representations in local Arthur packets.
\end{abstract}

\tableofcontents

\section{Introduction}
Let $F$ be a non-Archimedean field of characteristic zero. Let 
Let $\mathrm{G}_n=\Sp_{2n}, \SO_{2n+1}, \SO_{2n}^{\alpha}$ be quasi-split classical groups, where $\alpha$ is a square class in $F$, and let $G_n=\mathrm{G}_n(F)$. 
Here, we identify a square class with the corresponding quadratic character of the Weil group $W_F$ via the local class field theory. 
Then the Langlands dual groups are 
$$\widehat{\mathrm{G}}_n(\BC) = \SO_{2n+1}(\BC), \Sp_{2n}(\BC), \SO_{2n}(\BC).$$
Let ${}^L\mathrm{G}_n$ be the $L$-group of $G_n$,
$${}^L\mathrm{G}_n= 
\begin{cases}
\widehat{\mathrm{G}}_n(\BC) & \text{ when } \mathrm{G}_n=\Sp_{2n}, \SO_{2n+1},\\
\SO_{2n}(\BC) \rtimes W_F & \text{ when } \mathrm{G}_n=\SO_{2n}^{\alpha}.
\end{cases}
$$ 
In his fundamental work \cite{Art13}, Arthur introduced the local Arthur packets which are finite sets of representations of $G_n$, parameterized by local Arthur parameters. Local Arthur parameters are defined as
a direct sum of irreducible representations
$$\psi: W_F \times \SL_2(\mathbb{C}) \times \SL_2(\mathbb{C}) \rightarrow {}^L\mathrm{G}_n$$
\begin{equation}\label{lap}
  \psi = \bigoplus_{i=1}^r \phi_i \otimes S_{m_i} \otimes S_{n_i},  
\end{equation}
satisfying the following conditions:

(1) $\phi_i(W_F)$ is bounded and consists of semi-simple elements, and $\dim(\phi_i)=k_i$;

(2) the restrictions of $\psi$ to the two copies of $\SL_2(\mathbb{C})$ are analytic, $S_k$ is the $k$-dimensional irreducible representation of $\SL_2(\mathbb{C})$, and 
$$\sum_{i=1}^r k_im_in_i = N=N_n:= 
\begin{cases}
2n+1 & \text{ when } \mathrm{G}_n=\Sp_{2n},\\
2n & \text{ when } \mathrm{G}_n=\SO_{2n+1}, \SO_{2n}^{\alpha}.
\end{cases}
$$ 
The first copy of $\SL_2(\mathbb{C})$ is called the Deligne-$\SL_2(\mathbb{C})$, denoted by $\SL_2^D(\mathbb{C})$. The second copy of $\SL_2(\mathbb{C})$ is called the Arthur-$\SL_2(\mathbb{C})$, denoted by $\SL_2^A(\mathbb{C})$. 
Let $\OO_{\psi}^D$ and $\OO_{\psi}^A$ be the corresponding nilpotent orbits via restricting $\psi$ to the first and the second copy of $\SL_2(\mathbb{C})$, respectively (more precisely, see Definition \ref{def orbit}). 
We let $\Psi(G_n)$ denote the set of local Arthur parameters of $G_n$.
Assuming the Ramanujan conjecture, Arthur (\cite{Art13}) showed that these local Arthur packets characterize the local components of square-integrable automorphic representations.
Given a local Arthur parameter $\psi$ as in \eqref{lap}, the local Arthur packet is denoted by $\Pi_{\psi}$. An irreducible admissible representation $\pi$ of $G_n$ is called of Arthur type if it lies in a local Arthur packet.
As an application, Arthur proved the local Langlands correspondence for $G_n$. 

Given a local Arthur parameter $\psi$ as in \eqref{lap}, in a series of papers (\cite{Moe06a, Moe06b, Moe09, Moe10, Moe11}), M{\oe}glin explicitly constructed each local Arthur packet $\Pi_{\psi}$ and showed that it is of multiplicity free.
Then, Xu (\cite{Xu17}) gave an algorithm to determine whether the representations in M{\oe}glin's construction are nonzero, and 
Atobe (\cite{Ato20}) gave a refinement on the M{\oe}glin's construction, using the new derivatives introduced by himself and M\'inguez (\cite{AM20}), which makes it relatively easier to compute the $L$-data. 
Unlike local $L$-packets which are disjoint, local Arthur packets may have nontrivial intersections. 
Recently, Atobe (\cite{Ato22}), Hazeltine and the authors (\cite{HLL22}) independently studied the intersection problem of local Arthur packets for symplectic and split odd special orthogonal groups, and gave different algorithms to determine when an irreducible representation is of Arthur type and what are the local Arthur packets containing it. Note that understanding the intersection of local Arthur packets is considered as a key step towards the local  non-tempered Gan-Gross-Prasad problem (see \cite[Conjecture 7.1, Remark 7.3]{GGP20}).

Let $\mathrm{G}$ be a connected reductive group and $G=\mathrm{G}(F)$.
Given an irreducible representation $\pi$ of $G$, one important invariant is a set $\mathfrak{n}(\pi)$ which is defined to be all the $F$-rational nilpotent orbits $\mathcal{O}$ in the Lie algebra $\mathfrak{g}$ of $G$ such that the coefficient $c_{\mathcal{O}}(\pi)$ in the Harish-Chandra-Howe local expansion of the character $\Theta(\pi)$ of $\pi$ is nonzero (see \cite{HC78} and \cite{MW87}). 
Let $\mathfrak{n}^m(\pi)$ be the subset of $\mathfrak{n}(\pi)$ consisting of maximal nilpotent orbits, under the closure ordering of nilpotent orbits. Let $\ol{\mathfrak{n}}(\pi)$ and $\ol{\mathfrak{n}}^m(\pi)$ be the sets of corresponding nilpotent orbits over $\ol{F}$. $\ol{\mathfrak{n}}^m(\pi)$ is called the geometric wavefront set of $\pi$.
It is an interesting and long-standing question to study the structures of the sets $\mathfrak{n}(\pi)$,  $\mathfrak{n}^m(\pi)$, $\ol{\mathfrak{n}}(\pi)$ and $\ol{\mathfrak{n}}^m(\pi)$. 
For a long time, it is expected that the geometric wavefront set is a singleton. However, recently, Tsai (\cite{Tsa22}) constructed provides examples showing that the geometric wavefront set may not always be a singleton. Hence, the geometric wavefront set is still very complicated and hard to compute in general. 
Recenlty, Okada (\cite{Oka21}) introduced new invariants, the canonical unramified wavefront sets for representations $\pi$ of depth-0, denoted by $\underline{\mathrm{WF}}(\pi)$, which facilitate a lot the computation of the geometric wavefront sets for depth-0, especially unipotent  representations of p-adic reductive groups (\cite{CMO21, CMO22, CMO23}). In particular, in \cite{CMO22}, using the canonical unramified wavefront sets $\underline{\mathrm{WF}}(\pi)$, Ciubotaru, Mason-Brown, and Okada gave a new characterization of the local Arthur packets corresponding to {\it basic} local Arthur parameters, i.e., those are trivial on $W_F \times \SL^D_2(\mathbb{C})$, for connected reductive groups, inner to split. More precisely, 

\begin{thm}[{\cite[Theorem 3.0.3]{CMO22}}]\label{intro thm 1}
Let $\mathrm{G}$ be a connected reductive group, inner to split, $G=\mathrm{G}(F)$. Assume that there is a local Arthur packets theory for $G$ as conjectured in \cite[Conjecture 6.1]{Art89} and the residue field of $F$ has sufficiently large characteristic.
   Let $\psi$ be a basic local Arthur parameter of $G$ and denote by $\lambda$ the (real) infinitesimal parameter associated with $\psi$. Then the local Arthur packet corresponding to $\psi$ can be characterized as
   \begin{equation}\label{intro equ 1}
       \Pi_{\psi}= \{ \pi \in \Pi(G)_{\lambda}\ | \ \underline{\mathrm{WF}}(\pi) \leq d_{A}(\OO_{\psi}^A,1) \}. 
   \end{equation}
Here $\Pi(G)_{\lambda}$ consists of representations of $G$ with infinitesimal parameter $\lambda$ and $d_{A}$ is the Achar's duality map defined in \cite[\S 4]{Ach03}.
$\OO_{\psi}^A$ is defined similarly as above. 
\end{thm}

An interesting phenomenon discovered by Ciubotaru, Mason-Brown, and Okada in \cite{CMO22} is that if replacing the canonical unramified wavefront set $\underline{\mathrm{WF}}(\pi)$ by the geometric wavefront set $\ol{\mathfrak{n}}^m(\pi)$ and replacing Achar's duality $d_A$ by the Barbasch-Vogan duality $d_{BV}$, then 
the right hand side of \eqref{intro equ 1} becomes much larger than $\Pi_{\psi}$. 
Here, the Barbasch-Vogan duality $d_{BV}$ is between the nilpotent orbits of $\widehat{\mathfrak{g}}(\BC) $ and $\mathfrak{g}(\BC)$ (see \cite{Spa82, BV85, Lus84, Ach03} and \S \ref{sec BV dual} for details).
Inspired by the case of real reductive groups, they conjectured that it would be a union of local Arthur packets as follows. 

\begin{conj}[{\cite[Conjecture 3.1.2]{CMO22}}]\label{conj weak packet intro}
Let $\mathrm{G}$ be a connected reductive group and $G=\mathrm{G}(F)$. Assume that there is a local Arthur packets theory for $G$ as conjectured in \cite[Conjecture 6.1]{Art89}.
Let $\psi$ be a basic local Arthur parameter of $G$ and denote $\lambda$ the (real) infinitesimal parameter associated with $\psi$. Then the weak local Arthur packet defined by
\begin{equation}\label{eq weak packet leq intro}
    \Pi_{\psi}^{\textrm{Weak}}:= \{ \pi \in \Pi(G)_{\lambda}\ | \ \WF(\pi) \leq d_{BV}(\OO_{\psi}^A) \}
\end{equation}
is a union of local Arthur packets.
\end{conj}

In this paper, we prove this conjecture for split classical groups with the assumption of the residue field characteristic of $F$ being large. We also generalize the definition of weak local Arthur packets to general representations of 
Arthur type (that is lying in some local Arthur packets) and prove the analogue results assuming Jiang's Conjecture on wavefront sets of representations in local Arthur packets as follows. 

\begin{conj}[Jiang's Conjecture, \cite{Jia14}]\label{conj Jiang}
Let $\mathrm{G}$ be a connected reductive group and $G=\mathrm{G}(F)$. 
Assume that there is a local Arthur packets theory for $G$ as conjectured in \cite[Conjecture 6.1]{Art89}.
Let $\psi$ be a local Arthur parameter of $G$, and $\Pi_\psi$ be the local Arthur packet associated with $\psi$. Then the following holds.
\begin{enumerate}
    \item [(i)] For any $\pi\in \Pi_{\psi}$, any nilpotent orbit $\OO$ in $\WF(\pi)$ has the property that
    $\OO \leq d_{BV}(\OO_{\psi}^A). $
    \item [(ii)] There exists at least one member $\pi \in \Pi_{\psi}$ having the property that $d_{BV}(\OO_{\psi}^A) \in 
    \WF(\pi)$.
\end{enumerate}
\end{conj}

Jiang's conjecture describes the connection between the structures of local Arthur parameters and the geometric wavefront sets of representations in local Arthur packets. It is a natural generalization of Shahidi's conjecture which says that tempered $L$-packets of quasi-split groups have generic members, and enhanced Shahidi's conjecture which says that a local Arthur packet of a quasi-split group has a generic member if and only if it is tempered. 
There has been many progresses on Jiang's conjecture, see \cite{LS23}, \cite{HLLS23} for more details. In particular, combining the results of \cite{CMO21, CMO22, CMO23, HLLZ22, HLLS23}, Jiang's conjecture is true for all local Arthur parameters of split $\SO_{2n+1}$ and $\Sp_{2n}$ which are trivial on $W_F$. More precisely, 

\begin{thm}[{\cite[Theorem 7.4]{HLLS23}}]\label{thm Jiang B C}
    Assume that the residue field of $F$ has sufficiently large characteristic. Conjecture \ref{conj Jiang} holds for any local Arthur parameter $\psi$ of split $\SO_{2n+1}(F)$ and $\Sp_{2n}(F)$ whose restriction to $W_F$ is trivial.
\end{thm}

Let $\Lambda(G_n)$ be the set of infinitesimal parameters of $G_n$.
For $\lambda \in \Lambda(G_n)$, we also let $\Phi(G_n)_{\lambda}$ (resp. $\Psi(G_n)_{\lambda}$) be the set of $L$-parameters (resp. local Arthur parameters) of $G_n$ whose associated infinitesimal parameter is $\lambda$.

Now we can state our main result.

\begin{thm}[{Theorem \ref{thm main}}]\label{thm main intro}
Let $G_n$ be the split group $\SO_{2n+1}(F)$, $\Sp_{2n}(F)$ or $\SO_{2n}(F)$. Assume that the residue field of $F$ has sufficiently large characteristic.
\begin{enumerate}
    \item [(a)] 
    For any basic local Arthur parameter $\psi_0$ of $G_n$, the weak local Arthur packet $\Pi_{\psi_0}^{\textrm{Weak}}$ is contained in a union of local Arthur packets
        \begin{align*}
        \Pi_{\psi_0}^{\textrm{Weak}}\subseteq \bigcup_{ \psi \in (d_{BV})_{\Psi(G_n)_{\lambda}}^{-1}(\OO')} \Pi_{\psi},
    \end{align*}
     where $\lambda= \lambda_{\psi^0}$, $\OO'=d_{BV}(\OO_{\psi_0}^A)$, and 
     \[(d_{BV})^{-1}_{\Psi(G_n)_{\lambda}}(\OO'):=\{ \psi \in \Psi(G_n)_{\lambda} \ | \ d_{BV}(\OO_{\psi}^A)=\OO' \}.\]
     \item [(b)] Moreover, assume Conjecture \ref{conj Jiang}(i) holds for any $\psi \in \Psi(G_n)$ whose restriction to $W_F$ is trivial. Then we have the other direction of containment
     \begin{align*}
        \Pi_{\psi_0}^{\textrm{Weak}}\supseteq \bigcup_{ \psi \in (d_{BV})_{\Psi(G_n)_{\lambda}}^{-1}(\OO')} \Pi_{\psi},
    \end{align*}
    which proves Conjecture \ref{conj weak packet intro}. In particular, by Theorem \ref{thm Jiang B C}, Conjecture \ref{conj weak packet intro} holds for split $\SO_{2n+1}(F)$ and $\Sp_{2n}(F)$ without the assumption of Conjecture \ref{conj Jiang}(i). 
\end{enumerate}
\end{thm}

We remark that Part (a) of above theorem proves that weak local Arthur packets $\Pi_{\psi_0}^{\textrm{Weak}}$ consist of unitary representations without any assumption. This proves \cite[Conejcture 3.1.3]{CMO22} (see Theorem \ref{thm CMO unitary}). We need to make the residue field characteristic assumption since we need to make use of the result in \cite[Theorem 1.4.1]{CMO23}. We also remark that the same proof works for inner forms of split group $\SO_{2n+1}(F)$, $\Sp_{2n}(F)$ or $\SO_{2n}(F)$ once Arthur's theory on the local Arthur packets is developed. Especially, in the proof, we need the results that $\Pi_{\phi_{\psi}} \subseteq \Pi_{\psi}$ and $\Pi_{\widehat{\psi}}= \{\widehat{\pi} \ | \ \pi \in \Pi_{\psi}\}$.

The following proposition plays a key role in the proof of Theorem \ref{thm main intro} Part (a).
For any local $L$-parameter $\phi$, we define $\OO_{\phi}$ to be the corresponding nilpotent orbit via restricting $\phi$ to the $\mathrm{SL
}_2(\mathbb{C})$.

\begin{prop}[Proposition \ref{prop fiber tempered nilpotent orbit}]\label{prop fiber L-par intro}
    Let $\lambda \in \Lambda(G_n)$. Suppose the unique open $L$-parameter $\phi^0$ in $\Phi(G_n)_{\lambda}$ is tempered and let $\OO'= d_{BV}(\OO_{\phi^0})$. Then any $L$-parameter in 
    $$ (d_{BV})^{-1}_{\Phi(G_n)_{\lambda}}(\OO'):=\{ \phi \in \Phi(G_n)_{\lambda} \ | \ d_{BV}(\OO_{\phi})=\OO' \}$$
    is of Arthur type.
\end{prop}

The proof of this proposition is based on the explicit description of $d_{BV}^{-1}(\OO'):= \{\OO \ | \ d_{BV}(\OO)=\OO'\}$ studied in \cite{LLS23} (see Lemma \ref{lemma nec}). 
Proposition \ref{prop fiber L-par intro} has its own interests. Conjecture \ref{conj weak packet intro} may not imply analogous results in Proposition \ref{prop fiber L-par intro} for general groups.

We remark that if we replace the $\WF(\pi)$ by $d_{BV}(\OO_{\phi_{\widehat{\pi}}})$ in the definition of weak local Arthur packet \eqref{eq weak packet leq intro}, then we can generalize Theorem \ref{thm main intro} to anti-tempered local Arthur parameters, which is not necessarily basic.
More precisely, we have the following. Note that in this case we do not need the characteristic assumption of the residue field of $F$.

\begin{thm}[Theorem \ref{thm main 2}]
\label{thm main 2 intro}
Let $G_n$ be the split group $\SO_{2n+1}(F)$, $\Sp_{2n}(F)$ or $\SO_{2n}(F)$. For any anti-tempered local Arthur parameter $\psi_0$, we denote $\lambda:=\lambda_{\psi_0}$ and $\OO':= d_{BV}(\OO_{\psi_0}^{A})$. Consider the set of representations
\[ \Pi_{\psi_0}^{\textrm{Weak}}:= \{ \pi \in \Pi(G_n)_{\lambda} \ | \ d_{BV}(\OO_{\phi_{\widehat{\pi}}}) \leq d_{BV}(\OO_{\psi_0}^{A})\}.\]
We have an inclusion
\begin{align*}
    \Pi_{\psi_0}^{\textrm{Weak}} \subseteq \bigcup_{ \psi \in (d_{BV})_{\Psi(G_n)_{\lambda}}^{-1}(\OO')} \Pi_{\psi}.
\end{align*}
Moreover, the above inclusion is an equality if $\OO_{\phi_{\pi}} \geq \OO_{\phi_{\psi}}$ for any $\psi \in \Psi(G_n)_{\lambda}$ and $\pi \in \Pi_{\psi}$, which has already been verified for the split groups $\SO_{2n+1}(F)$ and $\Sp_{2n}(F)$ in \cite[Theorem 1.15, Corollary 4.12(2)]{HLLZ22}.
\end{thm}

Next, we discuss the generalization of the weak local Arthur packets beyond the basic local Arthur parameters. 
Let $\mathrm{G}$ be a connected reductive group and $G=\mathrm{G}(F)$. 
Assume that there is a local Arthur packets theory for $G$ as conjectured in \cite[Conjecture 6.1]{Art89}.
A first generalization would be as follows
\begin{equation}\label{intro equ 2 psi}
    \Pi_{\psi}^{\textrm{WF}}:= \{ \pi \in \Pi(G)_{\lambda}\ | \ \WF(\pi) \leq d_{BV}(\OO_{\psi}^A) \},
\end{equation}
where $\psi$ is any local Arthur parameter and $\lambda=\lambda_{\psi}$. 
Or, more generally, 
\begin{equation}\label{intro equ 2}
\Pi_{\OO',\lambda}^{\textrm{WF}}:=  \{ \pi \in \Pi(G)_{\lambda} \ | \ \WF(\pi) \leq \OO'  \},
\end{equation}
for any nilpotent orbit $\OO'$ of $\RG(\BC)$ and any infinitesimal parameter $\lambda$ of $G$. However, 
these sets may not always be a union of local Arthur packets since they may contain representations not of Arthur type, see Example \ref{sec5 exam 1}. A natural modification of
\eqref{intro equ 2 psi} or \eqref{intro equ 2} is to add the condition of Arthur type as follows
\begin{align}\label{eq generalization leq A type 1 intro}
     \Pi_{\psi}^{\textrm{WF,A}}:=  \{ \pi \in \Pi(G)_{\lambda} \text{ of Arthur type} \ | \ \WF(\pi) \leq d_{BV}(\OO_{\psi}^A)  \},
\end{align}
where $\psi$ is any local Arthur parameter and $\lambda=\lambda_{\psi}$, 
or, more generally, 
\begin{align}\label{eq generalization leq A type 2 intro}
    \Pi_{\OO',\lambda}^{\textrm{WF,A}} :=  \{ \pi \in \Pi(G)_{\lambda} \text{ of Arthur type} \ | \ \WF(\pi) \leq \OO'  \},
\end{align}
for any nilpotent orbit $\OO'$ of $\RG(\BC)$ and any infinitesimal parameter $\lambda$ of $G$.
However, these weak local Arthur packets may still not be a union of local Arthur packets, see Example \ref{sec5 exam 2}.

Now, let us focus on the set $\Pi_{\OO',\lambda}^{\textrm{WF,A}}$. Taking Conjecture \ref{conj Jiang} into consideration, assuming Conjecture \ref{conj Jiang}(i), we have
\begin{align}\label{eq containment generalization leq A type intro}
     \Pi_{\OO',\lambda}^{\textrm{WF,A}} \supseteq \bigcup_{ \substack{\psi \in \Psi(G_n)_{\lambda},\  d_{BV}(\OO_{\psi}^A) \leq \OO'} } \Pi_{\psi}.
\end{align}
Moreover, assuming Conjecture \ref{conj Jiang}(ii), the right hand side of \eqref{eq containment generalization leq A type intro} is exactly the union of all local Arthur packets contained in $ \Pi_{\OO',\lambda}^{\textrm{WF,A}}$. More precisely, we have the following

\begin{prop}[Proposition \ref{prop generalization}]
\label{prop generalization intro}
Let $\mathrm{G}$ be a connected reductive group and $G=\mathrm{G}(F)$. 
Assume that there is a local Arthur packets theory for $G$ as conjectured in \cite[Conjecture 6.1]{Art89}.
    Assume Conjecture \ref{conj Jiang} holds for the group $G=\RG(F)$. For any nilpotent orbit $\OO'$ of $\RG(\BC)$ and any infinitesimal parameter $\lambda$ of $G$, we define the weak local Arthur packet as follows
    \begin{align*}
        \Pi_{\OO',\lambda}^{\textrm{Weak}}:=\{\pi \in \Pi(G)_{\lambda} \textrm{ of Arthur type}\ | \ \exists \psi \in \Psi(\pi) \text{ such that }d_{BV}(\OO_{\psi}^{A}) \leq \OO'\},
    \end{align*}
    where $\Psi(\pi):=\{\psi \in \Psi(G)\ | \ \pi \in \Pi_{\psi}\}.$
    Then 
    \[\Pi_{\OO',\lambda}^{\textrm{Weak}}=   \bigcup_{\psi \in \Psi(G)_{\lambda},\ \Pi_{\psi}\subseteq \Pi_{\OO',\lambda}^{\textrm{WF,A}} } \Pi_{\psi} \ \subseteq\  \Pi_{\OO',\lambda}^{\textrm{WF,A}}, \]
    where the containment can be strict.
\end{prop}

The set $\Pi_{\OO',\lambda}^{\textrm{Weak}}$ can be regarded as a natural generalization of $\Pi_{\psi_0}^{\textrm{Weak}}$ and reveals the implications of Conjecture \ref{conj Jiang}.
Indeed, assuming Conjecture 
\ref{conj Jiang}, one can see from Proposition \ref{prop generalization intro}
that 
$\Pi_{\OO',\lambda}^{\textrm{Weak}}$ is 
the maximal subset of $\Pi(G)_{\lambda}$ with the following properties.
\begin{enumerate}
    \item [$\oldbullet$]  $\Pi_{\OO',\lambda}^{\textrm{Weak}} \subseteq \{ \pi \in \Pi(G)_{\lambda}\ | \ \WF(\pi) \leq \OO' \}$.
    \item [$\oldbullet$] $\Pi_{\OO',\lambda}^{\textrm{Weak}}$ is a union of local Arthur packets.
\end{enumerate} 
Hence, if Conjecture \ref{conj weak packet intro} holds for a basic local Arthur parameter $\psi_0$ of $G$, then
\[ \Pi_{\psi_0}^{\textrm{Weak}}=\Pi_{\OO',\lambda}^{\textrm{Weak}},\]
where $\OO'= d_{BV}(\OO_{\psi_0}^A)$ and $\lambda=\lambda_{\psi_0}$.

Following is the structure of the paper. In \S 
\ref{sec: preliminaries}, we introduce necessary preliminaries. In \S \ref{sec: fibers}, we  
recall certain results from \cite{LLS23} on the fibers of the Barbasch-Vogan duality and prove the key Proposition \ref{prop fiber L-par intro}. In \S \ref{sec: main thm}, we prove our main result Theorem \ref{thm main intro} on the weak local Arthur packets Conjecture \ref{conj weak packet intro} and its variant Theorem \ref{thm main 2 intro}. In \S \ref{sec: generalization}, we generalize the definition of weak local Arthur packets and prove Proposition \ref{prop generalization intro}.

\subsection*{Acknowledgements} 
The authors would like to thank Professor Dihua Jiang and Professor Freydoon Shahidi for their interests and constant support. The authors also would like to thank Cheng-Chiang Tsai for helpful communication. Part of this work was done when the second-named author was visiting the National Center for Theoretical Science, for which he thanks for their hospitality.


\section{Preliminaries}\label{sec: preliminaries}

\subsection{\texorpdfstring{Partitions and nilpotent orbits of $\mathfrak{g}_n(\BC)$}{}}
In this subsection, we recall the basic notation for partitions and the correspondence between nilpotent orbits of $\mathfrak{g}_n(\BC)$ and partitions.

First, we denote the set of partitions of $n$ by $\mathcal{P}(n)$. We express a partition $\underline{p}\in \mathcal{P}(n)$ in one of the following forms.
\begin{enumerate}
    \item [(i)] $\underline{p}=[p_1,\dots, p_N]$, such that $p_i$'s are non-increasing and $\sum_{i=1}^N p_i= n$. We denote the the \emph{length} of $\underline{p}$ by $l(\underline{p})= |\{1 \leq i \leq N \ | \ p_i >0\}|$.
    \item [(ii)] $\underline{p}=[p_1^{r_1},\dots, p_N^{r_N}]$, such that $p_i$'s are decreasing and $\sum_{i=1}^N r_i p_i=n$. We assume $r_i>0$ unless specified.
\end{enumerate}
Also, we denote $|\underline{p}|=n$ if $\underline{p} \in \mathcal{P}(n)$.

Next, we recall the definitions for partitions of type $B$, $C$ and $D$.
\begin{defn}
For $\epsilon\in \{ \pm 1\}$, we define
\[ \mathcal{P}_{\epsilon}(n)= \{ [p_1^{r_1},\dots, p_N^{r_N}]\in \mathcal{P}(n) \ | \ r_i \text{ is even for all }p_i \text{ with } (-1)^{p_i}=\epsilon \}. \]
Then we say
\begin{enumerate}
    \item $\underline{p}\in \mathcal{P}(n)$ is of type $B$ if $n$ is odd and $\underline{p}\in \mathcal{P}_{1}(n)$.
    \item $\underline{p}\in \mathcal{P}(n)$ is of type $C$ if $n$ is even and $\underline{p}\in \mathcal{P}_{-1}(n)$.
    \item $\underline{p}\in \mathcal{P}(n)$ is of type $D$ if $n$ is even and $\underline{p}\in \mathcal{P}_{1}(n)$.
\end{enumerate}
We denote $\mathcal{P}_X(n)$ the set of partitions of $n$ of type $X$.
\end{defn}

Denote the set of nilpotent orbits of $\mathfrak{so}_{2n+1}(\BC)$, $\mathfrak{sp}_{2n}(\BC)$ and $\mathfrak{so}_{2n}(\BC)$ by $\mathcal{N}_B(2n+1),$ $\mathcal{N}_C(2n)$ and $\mathcal{N}_D(2n)$ respectively. Also, we denote
 \[ \mathcal{N}_B =\bigcup_{n \geq 0} \mathcal{N}_{B}(2n+1),\ \mathcal{N}_C =\bigcup_{n \geq 0} \mathcal{N}_{C}(2n),\ \mathcal{N}_D =\bigcup_{n \geq 0} \mathcal{N}_{D}(2n).  \]
For $(X,N)\in \{ (B,2n+1 ),(C,2n ), (D,2n ) \}$, there is a surjection
\begin{align*}
    \mathcal{N}_X(N) & \longrightarrow \mathcal{P}_{X}(N)\\
    \OO& \longmapsto \underline{p}_{\OO}.
\end{align*}
The fiber of $\underline{p}= [p_1^{m_1},\dots,p_r^{m_r}] \in \mathcal{P}_X(N)$ under this map is a singleton, which we denote by $\{\OO_{\underline{p}}\}$, except when $\underline{p}$ is ``very even"; i.e., $\underline{p}$ is of type $D$ and $p_i$'s are all even. When $\underline{p}$ is very even, the fiber consists of two nilpotent orbits, which we denote by $\OO_{\underline{p}}^{I}$ and $\OO_{\underline{p}}^{II}$.

The surjection $\OO \mapsto \underline{p}_{\OO}$ carries the closure ordering on $\mathcal{N}_X(N)$ to the dominance ordering on $\mathcal{P}_X(N)$ in the sense that $\OO > \OO'$ if and only if $\underline{p}_{\OO}> \underline{p}_{\OO'}$. Note that when $\underline{p}$ is very even, $\OO_{\underline{p}}^{I}$ and $\OO_{\underline{p}}^{II}$ are not comparable.

\subsection{Barbasch-Vogan duality}\label{sec BV dual}
In this subsection, following \cite{Spa82, BV85, Lus84, Ach03}, we introduce several operations on the set of partitions, and then use them to describe the definition of the Barbasch-Vogan duality on the level of partitions and nilpotent orbits.

First, we need the following operations to construct or decompose partitions.
\begin{defn} Suppose $\underline{p} \in \mathcal{P}(n_1)$ and $\underline{q} \in \mathcal{P}(n_2)$.
\begin{enumerate}
    \item [(i)] Write $\underline{p}=[p_1^{r_1},\dots, p_N^{r_N}]$ and $\underline{q}=[p_1^{s_1},\dots, p_N^{s_N}]$, where we allow $r_i=0$ or $s_i=0$. Then we define
    \[ \underline{p}\sqcup \underline{q}= [p_1^{r_1+s_1},\dots, p_N^{r_N+s_N}]. \]
    \item [(ii)] Write $\underline{p}=[p_1,\dots, p_N]$, we define 
\begin{align*}
    \underline{p}^+&= [p_1+1,p_2,\dots, p_N] \in \mathcal{P}(n_1+1),\\
        \underline{p}^-&= [p_1,\dots, p_{N-1}, p_N-1] \in \mathcal{P}(n_1-1).
\end{align*}
\end{enumerate}
\end{defn}

We recall the definition of transpose (or conjugation) of partitions.
\begin{defn}\label{def transpose}
For $\underline{p}=[p_1,\dots, p_N]\in \mathcal{P}(n)$, we define $\underline{p}^{\ast}=[p_1^{\ast}, \dots, p_{N'}^{\ast}]\in \mathcal{P}(n)$ by
\[ p_i^{\ast}=|\{ j \ | \ p_j \geq i   \}|. \]
\end{defn}

Next, we recall the definition of collapse. Let $n$ be a positive integer and let $X=B$ if $n$ is odd and $X\in \{C,D\}$ if $n$ is even. For any $\underline{p}\in \mathcal{P}(n)$, there exists a unique maximal partition $\underline{p}_X\in \mathcal{P}(n)$ of type $X$ such that $\underline{p}_X \leq \underline{p}$. We denote $\underline{p}_X$ the \emph{$X$-collapse} of $\underline{p}$.

 Now we recall the definition of Barbasch-Vogan duality for partitions of type $X$. Following the notation in \cite{Ach03}, we shall omit the parentheses between the superscript and subscript. For example, we shall write $\underline{p}_{D} {}^{+} \underline{\vphantom{p}}_{B}  {}^{-\ast}$ instead of $((((\underline{p}_{D})^{+})_{B})^{-})^{\ast}$.

\begin{defn}

\begin{enumerate}
    \item [(i)]For $\underline{p}\in \mathcal{P}(2n+1)$ of type $B$, we define $d_{BV}(\underline{p}):= \underline{p}^{-}\underline{\vphantom{p}}_C{}^{\ast}$, which is in $\mathcal{P}(2n)$ of type $C$.
    \item [(ii)]For $\underline{p}\in \mathcal{P}(2n)$ of type $C$, we define $d_{BV}(\underline{p}):= \underline{p}^{+}\underline{\vphantom{p}}_B{}^{\ast}$, which is in $\mathcal{P}(2n+1)$ of type $B$.
    \item [(iii)] For $\underline{p}\in \mathcal{P}(2n)$ of type $D$, we define $d_{BV}(\underline{p}):= \underline{p}^{\ast}\underline{\vphantom{p}}_D$, which is in $\mathcal{P}(2n)$ of type $D$.
\end{enumerate}
\end{defn}

Finally, we recall the extension of the Barbasch-Vogan duality to the level of nilpotent orbits. If $\underline{p} \in \mathcal{P}_{D}(2n)$ is very even, then we define
\[d_{BV}(\OO_{\underline{p}}^I)=\begin{cases}
    \OO_{\underline{p}}^I& \text{ if }n \text{ is even,}\\
    \OO_{\underline{p}}^{II}& \text{ if }n \text{ is odd.}
\end{cases}\]
Otherwise, we define $d_{BV}(\OO_{\underline{p}})=\OO_{d_{BV}(\underline{p})} $. See \cite[Corollary 6.3.5]{CM93}. We say a nilpotent orbit or a partition is \emph{special} if it is in the image of the Barbasch-Vogan duality.

\subsection{Definition of parameters}\label{sec par}
In this subsection, we give the definition of $L$-parameters, local Arthur parameters and infinitesimal parameters of $G_n$ and related notations. We also allow $G_n$ to be $\GL_n(F)$ in this subsection.

First, we use the following definition of $L$-parameters and local Arthur parameters. 
\begin{defn}\label{def L-par}
  An $L$-parameter $[\phi]$ of $G_n$ is a $\widehat{G}_n(\BC)$-conjugacy class of an admissible homomorphism
 \[ \phi: W_F \times \SL_2(\BC) \to \widehat{G}_n(\BC).\]
 That is, $\phi$ is continuous, and
 \begin{enumerate}

\item the restriction of $\phi$ to $W_F$ consists of semi-simple elements;

\item the restriction of $\phi$ to $\SL_2(\BC)$ is a morphism of complex algebraic groups;

 \item if the image of $\phi$ is contained in the Levi subgroup of some parabolic subgroup $P$ of ${}^L G$, then $P$ is relevant for $G$ (see \cite[8.2(ii)]{Bor79} for notation).

\end{enumerate}
By abuse of notation, we don't distinguish $[\phi]$ and $\phi$.
\end{defn}

\begin{defn}\label{def A-par}
    A local Arthur parameter $[\psi]$ of $G_n$ is a $\widehat{G}(\BC)$-conjugacy class of a homomorphism
    \[ \psi: W_F \times \SL_2^D(\BC) \times \SL_2^A(\BC) \to \widehat{G}_n(\BC), \]
    such that
     \begin{enumerate}
    \item $\psi|_{W_F \times \SL_2(\BC)^D}$ is an $L$-parameter;
    \item the restriction of $\psi$ to $\SL_2^A(\BC)$ is a morphism of complex algebraic groups;
    \item $\psi|_{W_F}$ has bounded image.
\end{enumerate}
By abuse of notation, we don't distinguish $[\psi]$ and $\psi$. 
\end{defn}
For each local Arthur parameter $\psi \in \Psi(G_n)$, we may define another local Arthur parameter $\widehat{\psi}$ by swapping $\SL_2^{D}(\BC)$ and $\SL_2^{A}(\BC)$. Namely, the morphism $\widehat{\psi}$ is given by
\begin{align}\label{eq psi hat}
    \widehat{\psi}(w,x,y):= \psi(w,y,x).
\end{align}

For $w \in W_F$, we denote $d_{w}=\textrm{diag}(|w|^{1/2}, |w|^{-1/2}) \in \SL_2(\BC)$. Then for each $\psi \in \Psi(G_n)$, we may associate an $L$-parameter $\phi_{\psi}$ by
\[ \phi_{\psi}(w,x):= \psi(w, x, d_w).\]
It is proved in \cite{Art13} that the map $\psi \mapsto \phi_{\psi}$ is an injection. We say an $L$-parameter $\phi \in \Phi(G_n)$ is of \emph{Arthur type} if $\phi=\phi_{\psi}$ for some $\psi \in \Psi(G_n)$.

In \cite{Art13}, for each $L$-parameter $\phi$ (resp. local Arthur parameter $\psi$) of $G_n$, he constructed a finite set $\Pi_{\phi}$ (resp. $\Pi_{\psi}$) of $\Pi(G_n)$, called the $L$-packet (resp. local Arthur packet) of $\phi$ (resp. $\psi$). It is shown in \cite{Art13} that the local Arthur packet $\Pi_{\psi}$ contains the associated $L$-packet $\Pi_{\phi_{\psi}}$.

Similar to the assignment $\psi \mapsto \phi_{\psi}$, for each $L$-parameter $\phi$, we may associate a morphism $\lambda_{\phi}$ from $W_F$ to $ \widehat{G}_n(\BC)$ by
\[ \lambda_{\phi}(w):= \phi(w, d_w). \]
For $\psi \in \Psi(G_n)$, we shall denote $\lambda_{\psi}:= \lambda_{\phi_{\psi}}$ for short. This gives an \emph{infinitesimal parameter} of $G_n$ in the following sense.
\begin{defn}
    An infinitesimal parameter $[\lambda]$ of $G_n$ is a $\widehat{G}_n(\BC)$-conjugacy class of a continuous homomorphism
    \[ \lambda: W_F \to \widehat{G}_n, \]
    whose image consists of semi-simple elements. By abuse of notation, we don't distinguish $[\lambda]$ and $\lambda$. We denote $\Lambda(G_n)$ the set of infinitesimal parameters of $G_n$.
\end{defn}
It is shown in \cite{Moe06b, Moe09} that for any $\pi \in \Pi_{\psi}$, we have $\lambda_{\phi_{\pi}}= \lambda_{\psi}$.

 For each $L$-parameter $\phi $ and each local Arthur parameter $\psi$ of $G_n$, we associate nilpotent orbits $\OO_{\phi}$, $\OO_{\psi}^D$ and $\OO_{\psi}^{A}$ and partitions 
$\underline{p}(\phi)$, $\underline{p}^{D}(\psi)$ and $\underline{p}^A(\psi)$ as follows. 

\begin{defn}\label{def orbit}
    For $\phi \in \Phi(G_n)$ and $\psi \in \Psi(G_n)$, we define $\OO_{\phi}$ (resp. $\OO_{\psi}^D$, $\OO_{\psi}^{A}$) to be the nilpotent orbit of  $\widehat{\mathfrak{g}}_n(\BC)$ containing the element
    \[ d(\phi|_{\SL_2}) \left(\begin{pmatrix}
    0 &1 \\0 &0
\end{pmatrix}\right)\  \left( \text{resp. } d(\psi|_{\SL_2^{D}}) \left(\begin{pmatrix}
    0 &1 \\0 &0
\end{pmatrix}\right),\ d(\psi|_{\SL_2^{A}}) \left(\begin{pmatrix}
    0 &1 \\0 &0
\end{pmatrix}\right)  \right), \]
and define the partition $\underline{p}(\phi):= \underline{p}_{\OO_{\phi}}$ (resp. $\underline{p}^{D}(\psi):= \underline{p}_{\OO_{\psi}^D}$, $\underline{p}^A(\psi):= \underline{p}_{\OO_{\psi}^D}$). Note that $\OO_{\widehat{\psi}}^A= \OO_{\psi}^D= \OO_{\phi_{\psi}}$ and $\underline{p}^A(\widehat{\psi})= \underline{p}^D(\psi)= \underline{p}( \phi_{\psi})$.
\end{defn}

Fix a $\lambda \in \Lambda(G_n)$. There is a natural partial ordering $\geq_{C}$ on the set $\Phi(G_n)_{\lambda}$, which is induced from the closure ordering on the associated Vogan variety. See \cite[Definition 1.10]{HLLZ22} for details. In \cite[Corollary 4.12 (2)]{HLLZ22}, we show that $\phi_1 \geq_C \phi_2$ implies that $\underline{p}(\phi_1) \geq \underline{p}(\phi_2)$ when $G_n=\SO_{2n+1}(F)$ or $\Sp_{2n}(F)$. Indeed, the same proof for $\SO_{2n}(F)$ shows that $ \phi_1>_C \phi_2$ implies that $\underline{p}(\phi_1) > \underline{p}(\phi_2)$. We recall the following.

\begin{prop}\label{prop open par}
    Let $\lambda \in \Lambda(G_n)$. The following holds.
    \begin{enumerate}
        \item [(a)] There exist unique $\phi^0, \phi_0 \in \Phi(G_n)_{\lambda}$ such that for any $\phi\in \Phi(G_n)_{\lambda}$, the inequality holds
        \[ \underline{p}(\phi^0) \geq \underline{p}(\phi) \geq \underline{p}(\phi_0). \]
        We call $\phi^{0}$ (resp. $\phi_0$) the open (resp. closed) $L$-parameter of $\Phi(G_n)_{\lambda}$ .
        \item [(b)] $\Phi(G_n)_{\lambda}$ contains an $L$-parameter of Arthur type if and only if $\phi^0$ is tempered. Let $\psi^{0}=\phi^0 \otimes S_1$ so that $\phi^0= \phi_{\psi^0}$. Then $\phi_0$ is also of Arthur type with $\phi_0= \phi_{\widehat{\psi}^0}$.
    \end{enumerate}
\end{prop}
\begin{proof}
These statements follow from \cite[Lemma 6.2, 6.4]{HLLZ22}.
\end{proof}
The nilpotent orbits version of Part (a) of above proposition also holds since $\underline{p}_{\OO} > \underline{p}_{\OO'}$ if and only if $\OO> \OO'$.

We also need the following adjectives on parameters.
\begin{defn}
    Let $\phi \in \Phi(G_n),$ $\psi \in \Psi(G_n)$ and $\lambda \in \Lambda(G_n)$.
    \begin{enumerate}
        \item We say $\phi$ is tempered if $\phi|_{W_F}$ has bounded image. We say $\psi$ is tempered if $\phi_{\psi}$ is tempered, or equivalently, $\psi|_{\SL_2^A(\BC)}$ is trivial.
        \item We say $\phi$ (resp. $\psi$, $\lambda$) is unramified if $\phi|_{I_F}$ (resp. $\psi|_{I_F}$, $\lambda|_{I_F}$) is trivial. 
        \item We say an unramified infinitesimal parameter $\lambda$ is real if the eigenvalues of $\lambda(\text{Fr})$ are all real, where $\text{Fr}$ is any choice of Frobenius in $W_F$.
    \end{enumerate}
\end{defn}
Note that any tempered $L$-parameter $\phi$ is of Arthur type since $\phi=\phi_{\psi}$ where $\psi=\phi \otimes S_1$. (I.e., $ \psi(w,x,y):= \phi(w,x).$)

Finally, we recall the definition of Deligne-Langlands-Lusztig parameters. We assume $G$ is inner to split for simplicity.
\begin{defn}\label{def DLL par}
    A Deligne-Langlands-Lusztig parameter of the group $G$ is a $\widehat{G}(\BC)$-orbit of a  triple $(s,x,\rho)$, where
    \begin{enumerate}
        \item [$\oldbullet$] $s \in \widehat{G}(\BC)$ is semisimple;
        \item [$\oldbullet$] $x \in \widehat{\mathfrak{g}}(\BC)$ such that $\Ad(s)x=qx$;
        \item [$\oldbullet$]  $\rho$ is an irreducible representation of the component group of $\textrm{Cent}_{\widehat{G}(\BC)}(\{s,x\})$ that is trivial on $Z(\widehat{G}(\BC))$. 
    \end{enumerate}
    We denote $\Phi^{\mathrm{Lus}}(G)$ the set of Deligne-Langlands-Lusztig parameter of $G$.
\end{defn}

Let $\Pi^{\textrm{Lus}}(G)\subseteq \Pi(G)$ denote the subset of representations of unipotent cuspidal support defined in \cite{Lus95}. The following theorem is proved by \cite{KL87, Lus95, Lus02}. We refer the reader to \cite[Theorem 4.1.1]{CMO23} for details.
\begin{thm}[{Deligne-Langlands-Lusztig Correspondence}]
    There is a bijection
    \begin{align*}
        \Phi^{\textrm{Lus}}(G) & \longrightarrow \Pi^{\textrm{Lus}}(G) \\
        (s,x,\rho) & \longmapsto X(s,x,\rho),
    \end{align*}
    satisfying several desiderata (see \cite[Theorem 4.1.1]{CMO23} for details).
\end{thm}

We assume that the Local Langlands Correspondence is compatible with above correspondence (see \cite[\S 2.3]{AMS21} and \cite[\S 4]{AMS22}). Namely, the representation $\pi\in \Pi(G_n)$ is of unipotent cuspidal support if and only if its local $L$-parameter $\phi_{\pi}$ by is unramified. In this case, let
\begin{align*}
    s= \phi_{\pi}\left(\Fr, \begin{pmatrix}
    q^{1/2}& 0\\ 0& q^{-1/2} 
\end{pmatrix} \right) ,\ \ x= d(\phi_{\pi}|_{\SL_2})\left(\begin{pmatrix}
    0& 1\\ 0& 0 
\end{pmatrix} \right), 
\end{align*}
where $\Fr$ is any choice of Frobenius in $W_F$. Then $\pi= X(s,x,\rho)$ for some $\rho$.

\begin{remark}
    Indeed, let $s,x$ be defined above. There is an equality
    \[\textrm{Cent}_{\widehat{G}_n(\BC)}(\{s,x\})= \textrm{Cent}_{\widehat{G}_n(\BC)} (\textrm{im}(\phi)).\]
    Thus if $(\phi,\chi)$ is the enhanced $L$-parameter of $\pi$, then $\chi$ can also be viewed as an irreducible representation of the component group of $\textrm{Cent}_{\widehat{G}(\BC)}(\{s,x\})$ that is trivial on $Z(\widehat{G}(\BC))$. Though not needed in this paper, one may add the requirement that $\rho=\chi$ under the comparison.   
\end{remark}

\subsection{Aubert-Zelevinsky involution}\label{sec AZ}
Let $G$ be any connected reductive algebraic group defined over $F$. In \cite{Aub95}, she showed that for any representation $\pi$ of $\Pi(G)$, there exists $\varepsilon\in\{\pm 1\}$ such that
\begin{align*}
\widehat{\pi}:=\varepsilon\sum_P (-1)^{\mathrm{dim}(A_P)}[\mathrm{Ind}_{P}^{G_n}(\mathrm{Jac}_P(\pi))]
\end{align*}
gives an irreducible representation. Here the sum is in the Grothendieck group and is taking over all standard parabolic subgroups $P$ of $G_n$ and $A_P$ is the maximal split torus of the center of the Levi subgroup of $P.$ Moreover, the map $\pi \mapsto \widehat{\pi}$ is an involution on $\pi(G)$. We call $\widehat{\pi}$ the Aubert-Zelevinsky involution of $\pi.$ 

It is expected that the Aubert-Zelevinsky involution is compatible with the local Arthur packet $\Pi_{\psi}$ in the sense that 
\[ \Pi_{\widehat{\psi}}= \{ \widehat{\pi}\ | \ \pi \in \Pi_{\psi} \}.\]
(Recall that $\widehat{\psi}$ is defined by \eqref{eq psi hat}.) When $G=G_n$, this is discussed and proved in \cite[\S 7.1]{Art13} and \cite[\S A]{Xu17}. 

\section{Fibers of the Barbasch-Vogan duality}\label{sec: fibers}

In this section, we first recall certain results from \cite{LLS23} on the fibers of the Barbasch-Vogan duality. Then we prove a key result (Proposition \ref{prop fiber tempered nilpotent orbit} below) showing that certain $L$-parameters are of Arthur type, which plays an important role in the proof of the weak local Arthur packets conjecture next section. 

\subsection{Partitions and nilpotent orbits}\label{sec fiber}
Let $(X,X')\in \{(B,C),(C,B), (D,D)\}$. In this subsection, we describe the structure of the sets of partitions 
\[ d_{BV}^{-1}(\underline{\mathfrak{p}}):=\{\underline{p} \in \mathcal{P}_{X} \ | \ d_{BV}(\underline{p})= \underline{\mathfrak{p}} \},\]
for a special partition $\underline{\mathfrak{p}}\in \mathcal{P}_{X'}$. In \cite{LLS23}, joint with Shahidi, we give an explicit description of $d_{BV}^{-1}(\underline{\mathfrak{p}})$  and relate it with $d_{BV}^{-1}(\OO')$ (see Proposition \ref{prop nilpotent fiber} below). We recall certain results as follows.

Recall that when we write a partition $\underline{p}$ as $[p_1,\dots, p_r]$, we require $p_i$'s to be non-increasing. We set $p_t=0$ for any $t >r$ throughout this section.
Given a partition $\underline{p} \in \mathcal{P}_X$, the following lemma describes a necessary condition on $\underline{q}$ such that $\underline{p}\geq \underline{q}$ and $d_{BV}(\underline{p})= d_{BV}(\underline{q})$.

\begin{lemma}[{\cite[Lemma 3.5, Corollary 3.9]{LLS23}}]
\label{lemma nec}
    Let $X \in \{B,C,D\}$.
    Suppose $\underline{p}=[p_1,\dots, p_r], \underline{q}=[q_1 ,\dots, q_s] \in \mathcal{P}_X(n)$ satisfy that $\underline{p}\geq \underline{q}$ and $d_{BV}(\underline{p})= d_{BV}(\underline{q})$. Then there exists a sequence of pairs of positive integers $\{(x_i,y_i)\}_{i=1}^\alpha$ where
    \begin{enumerate}
        \item [(a)] $1 \leq x_i < y_i \leq r+1$;
        \item [(b)] $p_{x_{i}}= p_{x_i+1}+1=\cdots = p_{y_i-1}+1= p_{y_i}+2$, where we set $p_{r+1}=0$;
        \item [(c)] the sequence $( p_{x_1},\dots, p_{x_\alpha})$ is strictly decreasing;
    \end{enumerate}
    such that $\underline{q}$ can be obtained from $\underline{p}$ by replacing $\{ p_{x_i} , p_{y_i}\}_{i=1}^\alpha$ in $\underline{p}$ with $ \{ p_{x_i}-1, p_{y_i}+1\}_{i=1}^\alpha$. In particular, for any $1 \leq t \leq r$, we have
\begin{align}\label{eq drop}
    0\leq \sum_{z=1}^t p_z - \sum_{i=1}^t q_z \leq 1.
\end{align}\end{lemma}

We need the following refinement.

\begin{cor}\label{cor fiber partition refinement}
Let $X \in \{B,C,D\}$.
    Suppose $\underline{p}=[p_1,\dots, p_r], \underline{q}=[q_1 ,\dots, q_s] \in \mathcal{P}_X(n)$ satisfy that $\underline{p}\geq \underline{q}$ and $d_{BV}(\underline{p})= d_{BV}(\underline{q})$. Suppose further that $\underline{p}= \bigsqcup_{j \in J} \underline{p_j}$ and $\underline{q}= \bigsqcup_{j \in J} \underline{q_j}$ such that $\underline{p_j} \geq \underline{q_j}$ for all $j \in J$. Write $\underline{p_j}=[p_{j,1},\dots, p_{j,r_j}]$ and $\underline{q_j}=[q_{j,1},\dots, q_{j,s_j}]$. Then for each $j \in J$, there exists a sequence of pairs of positive integers $\{(x_{(j,k)},y_{(j,k)})\}_{k=1}^{\alpha_j}$ where
    \begin{enumerate}
        \item [(a)] $1 \leq x_{(j,k)} < y_{(j,k)} \leq r_j+1$;
        \item [(b)] $p_{j,x_{(j,k)}}= p_{j,x_{(j,k)}+1}+1=\cdots = p_{j,y_{(j,k)}-1}+1= p_{j,y_{(j,k)}}+2$, where we set $p_{j,r_j+1}=0$;
        \item [(c)] the sequence $( p_{j,x_{(j,1)}},\dots, p_{j,x_{(j,\alpha_j)}})$ is strictly decreasing;
    \end{enumerate}
    such that $\underline{q_j}$ can be obtained from $\underline{p_j}$ by replacing $\{ p_{j,x_{(j,k)}} , p_{j,y_{(j,k)}}\}_{k=1}^{\alpha_j}$ in $\underline{p_j}$ with $ \{ p_{j,x_{(j,k)}}-1, p_{j,y_{(j,k)}}+1\}_{k=1}^{\alpha_j}$. In particular, for any $1 \leq t \leq r_j$, we have
    \begin{align}\label{eq drop split}
        0\leq \sum_{z=1}^t p_{j,z} - \sum_{z=1}^t q_{j,z} \leq 1.
    \end{align}  
\end{cor}
\begin{proof}
The idea of the proof is similar to \cite[Lemma 5.1.1]{CMO23}. We shall use the following two statements whose proof can be found there.
\begin{enumerate}
    \item [(i)] If $\underline{p_1} \geq \underline{q_1}$ and $\underline{p_2} \geq \underline{q_2}$, then $\underline{p_1} \sqcup \underline{p_2} \geq \underline{q_1} \sqcup\underline{q_2}$.
    \item [(ii)] If $(\tau_1,\dots, \tau_{\gamma})$ is a sequence of  non-increasing integers, and $\sigma$ is any permutation of $\{1,\dots, \gamma\}$, then for any $1 \leq t \leq \gamma$,
    \[  \sum_{i=1}^t \tau_i  \geq  \sum_{i=1}^t \tau_{\sigma(i)}. \]
\end{enumerate}

Fix a $j \in J$. By considering the decomposition 
\[\underline{p}= \underline{p_j} \sqcup \left(\bigsqcup_{j' \in J \setminus \{j\}} \underline{p_{j'}}  \right),\ \underline{q}= \underline{q_j} \sqcup \left(\bigsqcup_{j' \in J \setminus \{j\}} \underline{q_{j'}}  \right),\]
we may assume $|J|=2$ and label $J=\{1,2\}$ with $j=1$. First, we prove the inequality \eqref{eq drop split}.

Recall that $\underline{q}=\underline{q_1} \sqcup \underline{q_2}$ and hence $s=s_1+s_2$. Let $f: \{ 1,\dots, s_1+s_2 \} \to \{1,2\}$ be a function such that 
\begin{align*}
    q_{i}= \begin{cases}
    q_{1, |\{ k \leq i \ | \ f(k)=1 \}|} & \text{ if }f(i)=1,\\
    q_{2, |\{ k \leq i \ | \ f(k)=2 \}|} & \text{ if }f(i)=2.
\end{cases}
\end{align*}
 Assume $r_1=s_1$ by adding zero to $\underline{p_1}$ if necessary, we consider a sequence of integers $\lambda=(\lambda_{1},\dots, \lambda_{s_1+s_2})$ given by
\[\lambda_i:=\begin{cases}
    p_{1, |\{ k \leq i \ | \ f(k)=1 \}|} & \text{ if }f(i)=1,\\
    q_{2, |\{ k \leq i \ | \ f(k)=2 \}|} & \text{ if }f(i)=2.
\end{cases}\]
Then there exists a permutation $\sigma$ of $\{1,\dots, r_1+s_2\}$ such that $\underline{p_1} \sqcup\underline{q_2}=[\lambda_{\sigma(1)},\dots, \lambda_{\sigma(r_1+s_2)} ].$ Note that we require the sequence $(\lambda_{\sigma(1)},\dots, \lambda_{\sigma(r_1+s_2)})$ to be non-increasing. Then for any $1 \leq t \leq r$, we have
\[ \sum_{z=1}^t p_z \geq \sum_{z=1}^t \lambda_{\sigma(z)} \geq \sum_{z=1}^t \lambda_{z} \geq \sum_{z=1}^t q_z. \]
Here the first inequality follows from $\underline{p}=\underline{p_1}\sqcup \underline{p_2} \geq \underline{p_1} \sqcup \underline{q_2}$ by (i), the second inequality follows from (ii), and the last inequality follows from $\underline{p_1} \geq \underline{q_1}$. As a consequence, for any $1 \leq t \leq r$, the inequality \eqref{eq drop} in Lemma \ref{lemma nec} gives
\[1 \geq \sum_{z=1}^t p_z -\sum_{z=1}^t q_z \geq \sum_{z=1}^t \lambda_z- \sum_{z=1}^t  q_z \geq \sum_{z=1}^t q_z- \sum_{z=1}^tq_z= 0.\]
Since 
\[\sum_{z=1}^t \lambda_z- \sum_{z=1}^t q_z=  \sum_{z=1}^{|\{k \leq t \ | \ f(k)=1\}|}p_{1,z}- \sum_{z=1}^{|\{k \leq t \ | \ f(k)=1\}|} q_{1,z},\]
this proves \eqref{eq drop split} by varying $t$.

As a consequence of \eqref{eq drop split}, for $j\in \{1,2\}$, there exits a sequence of pairs of positive integers $\{(x_{(j,k)},y_{(j,k)})\}_{k=1}^{\alpha_j}$ that satisfies Conditions (a) and (c) such that $\underline{q_j}$ can be obtained from $\underline{p_j}$ by replacing $\{ p_{j,x_{(j,k)}} , p_{j,y_{(j,k)}}\}_{k=1}^{\alpha_j}$ in $\underline{p_j}$ with $ \{ p_{j,x_{(j,k)}}-1, p_{j,y_{(j,k)}}+1\}_{k=1}^{\alpha_j}$. Also, since 
\[ \{p_{1,x_{(1,k)}}\}_{k=1}^{\alpha_1}\sqcup \{p_{2,x_{(2,k)}}\}_{k=1}^{\alpha_2} = \{p_{x_i}\}_{i=1}^{\alpha},\ \{p_{1,y_{(1,k)}}\}_{k=1}^{\alpha_1}\sqcup \{p_{2,y_{(2,k)}}\}_{k=1}^{\alpha_2} = \{p_{y_i}\}_{i=1}^{\alpha}, \]
where $\{(x_i,y_i)\}_{i=1}^{\alpha}$ is given by Lemma \ref{lemma nec}, Condition (b) also holds for $\{(x_{(j,k)},y_{(j,k)})\}_{k=1}^{\alpha_j}$. This completes the proof of the corollary.
\end{proof}

Finally, for $(X,X')\in \{ (B,C),(C,B),(D,D)\}$ and a special $\OO' \in \mathcal{N}_{X'}$, $d_{BV}^{-1}(\OO')$ can be related with $d_{BV}^{-1}(\underline{p}_{\OO'})$ in the following proposition.

\begin{prop}[{\cite[Proposition 2.10]{LLS23}}]
\label{prop nilpotent fiber}
    Let $(X,X')\in \{ (B,C),(C,B),(D,D)\}$. For each special $\OO' \in \mathcal{N}_{X'}$, we have the following.
    \begin{enumerate}
        \item [(a)] If $\underline{\mathfrak{p}}:=\underline{p}_{\OO'}$ is not very even of type $D$, then any $\underline{p} \in d_{BV}^{-1}(\underline{\mathfrak{p}})$ is not very even, and
        \[ d_{BV}^{-1}(\OO')= \{\OO_{\underline{p}} \ | \ \underline{p} \in d_{BV}^{-1}(\underline{\mathfrak{p}}) \}. \]
        \item [(b)] If $\underline{\mathfrak{p}}:=\underline{p}_{\OO'}$ is very even of type $D$, then
        \[ d_{BV}^{-1}(\OO')= \{  d_{BV}(\OO')\}, \]
        which is a singleton.
    \end{enumerate}
\end{prop}

\subsection{\texorpdfstring{$L$-parameters}{}}\label{sec fiber L-par}
In this subsection, we prove the following proposition showing that certain $L$-parameters are of Arthur type.

\begin{prop}\label{prop fiber tempered nilpotent orbit}
    Let $G_n$ be the split group $\SO_{2n+1}(F)$, $\Sp_{2n}(F)$ or $\SO_{2n}(F)$ and $\lambda \in \Lambda(G_n)$. Suppose the unique open $L$-parameter $\phi^0$ of $\Phi(G_n)_{\lambda}$ is tempered and denote $\OO':= d_{BV}(\OO_{\phi^0})$. Then any $L$-parameter $\phi$ in
\[(d_{BV})^{-1}_{\Phi(G_n)_{\lambda}}(\OO'):=\{ \phi \in \Phi(G_n)_{\lambda} \ | \ d_{BV}(\OO_{\phi})=\OO' \}\]
   is of Arthur type.
\end{prop}

When $G_n=\SO_{2n+1}(F)$ or $\Sp_{2n}(F)$, the map from the nilpotent orbits of $\widehat{\mathfrak{g}}_n(\BC)$ to partitions of the corresponding type is a bijection. Therefore, we have

\[ (d_{BV})^{-1}_{\Phi(G_n)_{\lambda}} (\OO')=(d_{BV})^{-1}_{\Phi(G_n)_{\lambda}}(\underline{p}_{\OO'}):=\{ \phi \in \Phi(G_n)_{\lambda} \ | \ d_{BV}(p(\phi))=\underline{p}_{\OO'} \}.\]
However, the above equality fails for $\SO_{2n}(F)$ when $\underline{\mathfrak{p}}=\underline{p}_{\OO'}$ is very even. In this case, we have 
\[(d_{BV})^{-1}_{\Phi(G_n)_{\lambda}}(\underline{\mathfrak{p}}) =(d_{BV})^{-1}_{\Phi(G_n)_{\lambda}}(\OO_{\underline{\mathfrak{p}}}^{I})\sqcup (d_{BV})_{\Phi(G_n)_{\lambda}}^{-1}(\OO_{\underline{\mathfrak{p}}}^{II}).   \]

In any case, Proposition \ref{prop fiber tempered nilpotent orbit} follows from its partition version. 

\begin{prop}\label{prop fiber tempered partition}
    Let $G_n$ be the split group $\SO_{2n+1}(F)$, $\Sp_{2n}(F)$ or $\SO_{2n}(F)$ and $\lambda \in \Lambda(G_n)$. Suppose the unique open $L$-parameter $\phi^0$ of $\Phi(G_n)_{\lambda}$ is tempered and denote $\underline{\mathfrak{p}}:= d_{BV}(\underline{p}(\phi^0))$. Then any $L$-parameter $\phi \in (d_{BV})^{-1}_{\Phi(G_n)_{\lambda}}(\underline{\mathfrak{p}})$ is of Arthur type.
\end{prop}

Now we give a more explicit description of the partitions $\underline{p}(\phi),$ $\underline{p}^{D}(\psi)$ and $\underline{p}^A(\psi)$. In the case of $\SO_{2n}(F)$, we don't distinguish $\phi$ (resp. $\psi$) and $\phi^{c}$ (resp. $\psi^{c}$), its outer conjugation of $\phi$, since they give the same partition. The parameter $\phi$ (resp. $\psi$) is uniquely determined by its composition with the embedding $\widehat{G}_n(\BC) \hookrightarrow \GL_N(\BC)$, and hence we may write
\begin{align}\label{eq decomp of par}
    \phi= \bigoplus_{i \in I} \rho_i \otimes S_{a_i},\ \  \psi= \bigoplus_{j \in J} \rho_j \otimes S_{a_j} \otimes S_{b_j},
\end{align}
where $\rho_i$'s and $\rho_j$'s are irreducible representations of $W_F$ and $S_a$ is the unique $a$-dimensional irreducible representation of $\SL_2(\BC)$. With this decomposition, we have
\[ \underline{p}(\phi)= \bigsqcup_{i \in I} [a_i^{\dim(\rho_i)}],\  \underline{p}^{D}(\psi)= \bigsqcup_{j \in J} [a_j^{\dim(\rho_j)\cdot b_j }],\ \underline{p}^{A}(\psi)= \bigsqcup_{j \in J} [b_j^{\dim(\rho_j)\cdot a_j }]. \]
Also, with the decomposition \eqref{eq decomp of par}, we have (again composing with $\widehat{G}_n(\BC) \hookrightarrow \GL_N(\BC)$)
\[ \lambda_{\phi}= \bigoplus_{i \in I} \left(\bigoplus_{k=0}^{a_i-1} \rho_i|\cdot|^{\half{a_i-1}-k}\right),\ \ \phi_{\psi}= \bigoplus_{j \in J} \left(\bigoplus_{k=0}^{b_j-1} \rho_j|\cdot|^{\half{b_j-1}-k}\right) \otimes S_{a_j}.\]
In particular, if $\phi$ is tempered and $\rho|\cdot|^{x} \subseteq \lambda_{\phi}$ where $\rho$ is irreducible with bounded image and $x \in \R$, then $x \in \half{1}\Z$.

In the following discussion, we shall treat $\phi$ (resp. $\psi$) as a self-dual $L$-parameter (resp. local Arthur parameter) of some $\GL_N(F)$. We often consider subrepresentation $\phi'$ of $\phi$. We denote $\phi\ominus\phi'$ the subrepresentation of $\phi$ such that $\phi= \phi' \oplus (\phi \ominus\phi').$

The following lemma is the key observation towards Proposition \ref{prop fiber tempered partition}.
\begin{lemma}\label{lem key observation tempered}
Let $\phi^0$ be a tempered self-dual $L$-parameter of $\GL_N(F)$ and denote $\lambda= \lambda_{\phi^0}$. Suppose $\phi \in \Phi(\GL_N(F))_{\lambda}$ is self-dual, and $\phi^0$, $\phi$ have decompositions 
\begin{align*}
    \phi^0= \phi^0_1\oplus \phi^{0}_2,\ \phi= \phi_1\oplus \rho|\cdot|^{z}\otimes S_{a} \oplus \phi_2,
\end{align*}
where 
\begin{enumerate}
    \item [(i)]$z \in \R$ and $\rho$ is irreducible with bounded image,
    \item [(ii)]$\lambda_{\phi_1^0}= \lambda_{\phi_1}$ is self-dual.
\end{enumerate}
Then the followings hold.
\begin{enumerate}
    \item [(a)] Suppose further that any irreducible summand of $\phi_2^0$ has dimension less than or equal to 
    \[\dim(\rho|\cdot|^z \otimes S_a)=\dim(\rho)\cdot a.\]
    Then $z=0$, and $ \phi_2^0 \supseteq \rho \otimes S_{a}$.
    \item [(b)] Suppose further that $z\neq 0$ and any irreducible summand of $\phi_2^0$ has dimension less than or equal to 
    \[\dim (\rho \otimes S_{a+1})=\dim(\rho)\cdot (a+1).\] Then $|z|=1/2$, and $\phi_2^0 \supseteq \rho \otimes S_{a+1}$.
\end{enumerate}
\end{lemma}
\begin{proof}
    Since $\lambda_{\phi^0}=\lambda= \lambda_{\phi}$ and $\lambda_{\phi_1^0}= \lambda_{\phi_1}$ by (ii), we see that
    \[ \lambda_{\phi_2^0} =\lambda_{\rho |\cdot|^{z} \otimes S_a \oplus \phi_2} \supseteq \lambda_{ \rho |\cdot|^{z} \otimes S_a }= \bigoplus_{i=0}^{a-1} \rho |\cdot|^{z+\half{a-1}-i}.\]
    This implies $z \in \half{1}\Z$.
    Since $\phi$ is self-dual, replacing $\rho$ by $\rho^{\vee}$, the dual of $\rho$, if necessary, we may assume $z \geq 0$. Then $\lambda_{\phi_2^0} \supseteq \rho |\cdot|^{z+\half{a-1}}$ implies that $\phi_2^0$ must contain an irreducible summand of the form $\rho \otimes S_{b}$ with $b \geq 2(z+ \half{a-1})+1$ (and $b \equiv 2z+ a \mod 2$). 
    
    For Part (a), the assumption gives
    \[ a \leq 2\left(z + \half{a-1}\right)+1 \leq b \leq a,\]
    which implies $z=0$ and $b=a$. For Part (b), the assumption implies $z \geq 1/2$ and
    \[ a+1 \leq 2\left(z + \half{a-1}\right)+1 \leq b \leq a+1.  \]
    Thus $b=a+1$ and $z=1/2$. This completes the proof of the lemma.   
\end{proof}

As a corollary, we prove a special case of Proposition \ref{prop fiber tempered partition} using Lemma \ref{lemma nec}.

\begin{cor}\label{cor special case A-type}
    Let $\phi^0$ be a tempered self-dual $L$-parameter of $\GL_N(F)$ of the form
    \begin{align*}
        \phi^0= \rho \otimes \left(\bigoplus_{ j=1 }^r S_{p_j}\right),
    \end{align*}
    where $\rho$ is one-dimensional, self-dual and the sequence $(p_1,\dots, p_r)$ is non-increasing. Denote $\lambda= \lambda_{\phi^0}$. Suppose $\phi \in \Phi(\GL_N(F))_{\lambda}$ is self-dual and $\underline{q}:=\underline{p}(\phi)$ can be obtained from $\underline{p}=\underline{p}(\phi^0)= [p_1,\dots, p_r]$ by replacing $\{p_{x_i}, p_{y_i}\}_{i=1}^{\alpha}$ in $\underline{p}$ with $\{p_{x_i}-1, p_{y_i}+1\}_{i=1}^{\alpha}$, where the sequence $\{(x_i,y_i)\}_{i=1}^{\alpha}$ satisfies Conditions (a), (b) and (c) in Lemma \ref{lemma nec} for $\underline{p}$. Then $\phi$ is of Arthur type. More explicitly, let $J:=\{1,\dots, r\} \setminus (\{x_1,\dots, x_{\alpha}\}\sqcup\{y_1,\dots, y_{\alpha}\})$, then $\phi=\phi_{\psi}$, where
    \[\psi= \bigoplus_{j \in J} \rho \otimes S_{p_j}\otimes S_1 \oplus \bigoplus_{i=1}^{\alpha} \rho \otimes S_{p_{x_i}-1}\otimes S_2. \]   
\end{cor}

\begin{proof}
We apply induction on $\alpha= \alpha(\phi, \phi^0)$. When $\alpha=0$, we have $\phi=\phi^0$ and the conclusion trivially holds. Now we assume $\alpha(\phi,\phi^0)=k>0$ and the conclusion is verified for any $\phi'$ with $\alpha(\phi',\phi^0)<k$.

Since $\phi^0$ contains $\rho\otimes (S_{p_1}\oplus \cdots\oplus S_{p_{i}})$ for $1 \leq i \leq x_1-1$, inductively applying Lemma \ref{lem key observation tempered}(a), we see that $\phi$ also contains $\rho \otimes (S_{p_1} \oplus\dots\oplus S_{p_{x_1-1}})$. Denote $\phi_1=\phi_1^0:= \rho \otimes (S_{p_1} \oplus\dots \oplus S_{p_{x_1-1}})$.

Denote $a:= p_{x_1-1}$ for simplicity and let $m$ be the multiplicity of $a$ in $\underline{p}(\phi)$, which is greater than or equal to 2 by assumption. Write
\begin{align*}
    \phi \ominus \phi_1 &\supseteq \bigoplus_{j=1}^m \rho|\cdot|^{z_j} \otimes S_{a},\\
    \phi^0 \ominus \phi_1^0 &=\rho\otimes S_{a+1}\oplus (\rho \otimes S_{a})^{\oplus (m-2)} \oplus \bigoplus_{i=y_1}^{r} \rho \otimes S_{p_i}.
\end{align*}
(The $\oplus (m-2)$ of $(\rho \otimes S_{a})^{\oplus (m-2)}$ means the multiplicity.) Since the multiplicity of $\rho|\cdot|^{\half{a-1}}$ in $\lambda_{\phi \ominus \phi_1}$ is the same as that of $\lambda_{\phi^0 \ominus \phi_1}$, which is exactly $m-2$, we see that at least 2 of $z_j$'s are not zero. Now we may write
\[ \phi^0= \phi_1^0 \oplus \phi_2^0,\ \phi=\phi_1 \oplus \rho|\cdot|^{z} \otimes S_{a} \oplus \phi_2 \]
for some $z \neq 0$. Applying Lemma \ref{lem key observation tempered}(b), we obtain that $|z|=1/2$. Since $\phi$ is self-dual, we may rewrite
\begin{align*}
    \phi^0&= (\phi_1^0\oplus \rho\otimes S_{a+1} \oplus \rho\otimes S_{a-1}) \oplus \widetilde{\phi_2^0},\\ 
    \phi&= (\phi_1\oplus \rho|\cdot|^{\half{1}}\otimes S_{a} \oplus \rho|\cdot|^{\half{-1}}\otimes S_a)  \oplus \left(\bigoplus_{j=1}^{m-2} \rho|\cdot|^{z_j} \otimes S_{a} \right)\oplus \widetilde{\phi_2}.
\end{align*}
Applying Lemma \ref{lem key observation tempered}(a) inductively again, we see that $\{z_j\}_{j=1}^{m-2}$ are all zero. Finally, let
\[ \phi':= (\phi_1 \oplus \rho\otimes S_{a+1} \oplus \rho\otimes S_{a-1} \oplus (\rho\otimes S_{a})^{\oplus (m-2)})  \oplus \widetilde{\phi_2}.  \]
It is not hard to see that $\alpha(\phi',\phi^0)= \alpha(\phi,\phi^0)-1$. Then the induction hypothesis for the pair $(\phi',\phi^0)$ gives $\phi'= \phi_{\psi'}$, where we let $J'=\{1,\dots, r\} \setminus ( \{x_2,\dots,x_{\alpha}\} \sqcup \{y_2,\dots, y_{\alpha}\})$, and
\[  \psi'= \bigoplus_{j \in  J'  } \rho \otimes S_{p_j}\otimes S_1 \oplus \bigoplus_{i=2 }^{\alpha} \rho \otimes S_{p_{x_i}-1}\otimes S_2. \]
Comparing $\phi$ and $\phi'$, we get the desired conclusion. This completes the proof of the corollary.  
\end{proof}

Now we prove Proposition \ref{prop fiber tempered partition}.

\begin{proof} We may write
\begin{align*}
    \phi^0= \bigoplus_{i \in I_{nsd}} (\rho_i \oplus \rho_i^{\vee}) \otimes (\oplus_{j\in J_i} S_{a_j} ) \oplus \bigoplus_{i \in I_{sd}} \rho_i \otimes (\oplus_{j\in J_i} S_{a_j} ),
\end{align*}
    where $\rho_i \not\cong \rho_i^{\vee}$ for $i \in I_{nsd}$ and $\rho_i \cong \rho_i^{\vee}$ for $i \in I_{sd}$, and $\rho_{i_1} \not\cong \rho_{i_2}\not\cong \rho_{i_1}^{\vee}$ for $i_1\neq i_2 \in I_{nsd}$ and $\rho_{i_1} \not\cong \rho_{i_2}$ for $i_1\neq i_2 \in I_{sd}$. Let $I=I_{nsd} \sqcup I_{sd}$ and for $i \in I$, denote
    \begin{align*}
        \phi_i^0:= \begin{cases}
             (\rho_i \oplus \rho_i^{\vee}) \otimes (\oplus_{j\in J_i} S_{a_j} ) & \text{ if }i \in I_{nsd},\\ 
             \rho_i \otimes (\oplus_{j\in J_i} S_{a_j} ) & \text{ if }i \in I_{sd},
        \end{cases}
    \end{align*}
    which we regard as a self-dual tempered $L$-parameter of $\GL_{n_i}(F)$ that factors through $\Sp_{n_i}(\BC)$ if $G_n=\SO_{2n+1}(F)$, and factors through $\mathrm{O}_{n_i}(\BC)$ if $G_n= \Sp_{2n}(F)$ or $\SO_{2n}(F)$.

    Suppose $\phi \in (d_{BV})^{-1}_{\Phi(G_n)_{\lambda}}(\underline{\mathfrak{p}})$, where $\underline{\mathfrak{p}}= d_{BV}( \underline{p}(\phi^0))$. Then we have a decomposition
    \[ \phi= \bigoplus_{i \in I} \phi_i,\]
    where $\lambda_{\phi_i}= \lambda_{\phi_i^0}$. Write
    \[\underline{p_i}:= \underline{p}(\phi_i^0)=[p_{i,1},\dots, p_{i,r_i}],\ \underline{q_i}= \underline{p}(\phi_i)=[q_{i,1},\dots, q_{i,s_i}].\]
    For each $i\in I$, we have $\underline{p_i} \geq \underline{q_i}$ by Proposition \ref{prop open par}(a) since $\phi_i^0$ is a tempered $L$-parameter of $\GL_{n_i}(F)$. Therefore, Corollary \ref{cor fiber partition refinement} implies that for any $i \in I$ and $1 \leq t \leq r_i$, 
    \begin{align}\label{eq can not drop too much}
        0 \leq \varepsilon(\underline{p_i}, \underline{q_i},t):=\sum_{j=1}^t p_{i,j} - \sum_{j=1}^t q_{i,j} \leq 1.
    \end{align}
    For $i \in I_{nsd}$, we may write $\underline{p_i}=\underline{p_i}'\sqcup \underline{p_i}'$ and $\underline{q_i}=\underline{q_i}'\sqcup \underline{q_i}'$. Therefore, for $1 \leq t \leq r_i/2$, if we define $\varepsilon(\underline{p_i}', \underline{q_i}',t)$ similarly, then
    \[\varepsilon(\underline{p_i}, \underline{q_i},2t)= 2 \cdot \varepsilon(\underline{p_i}', \underline{q_i}',t),\]
    and hence \eqref{eq can not drop too much} implies $\varepsilon(\underline{p_i}', \underline{q_i}',t)=0$. We conclude that $\underline{p_i}'= \underline{q_i}'$ and hence $\underline{p_i}=\underline{q_i}$. We take $\psi_i=\phi_i^0 \otimes S_1$ in this case.
    
    For $i \in I_{sd}$ such that $\dim(\rho)>1$, we have \[\underline{p_i}= \underbrace{\underline{p_i}'\sqcup \cdots \sqcup \underline{p_i}'}_{\dim(\rho) \textrm{ copies}},  \ \ \ \underline{q_i}= \underbrace{\underline{q_i}'\sqcup \cdots \sqcup \underline{q_i}'}_{\dim(\rho) \textrm{ copies}},\]
    and hence the same argument shows that $\underline{p_i}=\underline{q_i}$. We take $\psi_i=\phi_i^0 \otimes S_1$ in this case.

    Finally, for $i \in I_{sd}$ such that $\dim(\rho)=1$, Corollary \ref{cor fiber partition refinement} implies that the pair $(\phi_i^0,\phi_i)$ satisfies the assumption of Corollary \ref{cor special case A-type}, and hence there is a $\psi_i$ such that $\phi_i=\phi_{\psi_i}$. 
    
    In conclusion, we have constructed  $\psi= \bigoplus_{i \in I} \psi_i$ so that $\phi=\phi_{\psi}$. This completes the proof of the proposition.
\end{proof}

The proof above also gives the following corollary, which will be used in \cite{HLLS23}.

\begin{cor}
    Suppose $\lambda \in \Lambda(G_n)$ has a decomposition
    \[ \lambda=\bigoplus_{i \in I} \rho_i |\cdot|^{x_i},\]
    where $\rho_i$'s are irreducible representations of $W_F$ with bounded image, and $x_i \in \R$. Suppose further that $\rho_i$'s are either non-self-dual or $\dim(\rho_i)>1$. Then for any $\phi, \phi' \in \Phi(G_n)_{\lambda} $ such that $\underline{p}(\phi) \geq \underline{p}(\phi')$, we have $d_{BV}(\underline{p}(\phi))= d_{BV}(\underline{p}(\phi'))$ if and only if $\phi=\phi'$.
\end{cor}

We end this subsection by demonstrating an example that the set $(d_{BV})_{\Phi(G_n)_{\lambda_{\phi}}}^{-1}(d_{BV}( \underline{p}(\phi)))$ may contain an $L$-parameter not of Arthur type when $\phi$ is of Arthur type but not tempered.

\begin{exmp}
    Let $G_n=\SO_{11}(F)$. Consider 
    \begin{align*}
        \psi&= 1 \otimes S_{2} \otimes S_3+1 \otimes S_{4} \otimes S_1,\\
        \phi_{\psi}&= 1 \otimes S_4+ 1 \otimes S_2+ |\cdot|^{1/2} \otimes S_2 +|\cdot|^{-1/2} \otimes S_2.
    \end{align*}
    We have $\underline{p}^D(\psi)=\underline{p}(\phi_{\psi})=[4,2^3]$, and we set $\underline{\mathfrak{p}}= d_{BV}( [4,2^3])=[5,3,1^3]$. Then $(d_{BV})_{\Phi(G_n)_{\lambda}}^{-1}(\underline{\mathfrak{p}})=\{ \phi_{\psi}, \phi_1, \phi_2 \},$ where
    \begin{align*}
        \phi_1&= 1 \otimes S_4+ 1 \otimes S_2+ 1 \otimes S_2 +|\cdot|^{3/2} \otimes S_1+|\cdot|^{3/2} \otimes S_1,\\
        \phi_2&=1 \otimes S_4+ |\cdot|^{1} \otimes S_2+ |\cdot|^{-1} \otimes S_2 +|\cdot|^{1/2} \otimes S_1+|\cdot|^{-1/2} \otimes S_1,
    \end{align*}
    are both not of Arthur type.
\end{exmp}

\section{Unipotent representations with real infinitesimal character}\label{sec: main thm}
In this section, we apply the results in Section \ref{sec fiber L-par} to prove the conjecture for weak local Arthur packets of basic local Arthur parameters. We assume the residue field of $F$ has sufficiently large characteristic throughout the section.

First, we recall the definition of \emph{weak local Arthur packets} for \emph{basic local Arthur parameters}.

\begin{defn}
Let $\mathrm{G}$ be a connected reductive group and $G=\mathrm{G}(F)$. Assume that there is a local Arthur packets theory for $G$ as conjectured in \cite[Conjecture 6.1]{Art89}.
We say $\psi \in \Psi(G)$ is \emph{basic} if $\psi|_{W_F \times \SL_2^D(\BC)}$ is trivial. For each basic local Arthur parameter $\psi_0$ of $G$, we define the weak local Arthur packet associated with $\psi_0$ by
    \begin{align}\label{eq weak packet leq}
    \Pi_{\psi_0}^{\textrm{Weak}}:= \{ \pi \in \Pi(G)_{\lambda_{\psi_0}} \ | \ \WF(\pi) \leq d_{BV}( \OO^A_{\psi_0} )  \}.
    \end{align}
\end{defn}

We say an unramified infinitesimal parameter $\lambda$ is real if after composing with an embedding ${}^L\mathrm{G} \to {}^L\GL_N$, it decomposes as
\[ \lambda= \bigoplus_{i \in I} |\cdot|^{x_i},\]
where $x_i$ are all real numbers. For $\psi \in \Psi(G_n)$, it is not hard to see that $\lambda_{\psi}$ is unramified (resp. real) if and only if $\psi|_{I_F}$ (resp. $\psi|_{W_F}$) is trivial. In particular, if $\psi$ is basic, then $\lambda_{\psi}$ is real.

To proceed, we recall the following result from \cite{CMO23}.

\begin{thm}[{\cite[Theorem 1.4.1]{CMO23}}]\label{thm computation of WF}
Let $\RG$ be a connected reductive algebraic group defined over $F$, inner to split. Assume that the residue field of $F$ has sufficiently large characteristic. Suppose $\pi$ is a unipotent representation of $\RG(F)$ with real infinitesimal parameter. Then the geometric wavefront set of $\pi$ is a singleton, and
\begin{align}\label{eq CMO computation of WF}
    \WF(\pi)=\{ d_{BV}( \OO_{\phi_{\widehat{\pi}} } ) \},
\end{align}
where $\OO_{\phi_{\widehat{\pi}}}$ is a nilpotent orbit of $\widehat{\mathfrak{g}}(\BC) $ associated to the $L$-parameter $\phi_{\widehat{\pi}}$ (see Definition \ref{def orbit}).
\end{thm}

\begin{remark}
By \cite[Corollary 6.0.5]{CMO23}, we may replace the $\leq $ in \eqref{eq weak packet leq} by $=$. Moreover, applying Theorem \ref{thm computation of WF}, we have 
\begin{align*}
\Pi_{\psi_0}^{\textrm{Weak}}&= \{ \pi \in \Pi(G_n)_{\lambda} \ | \ \WF(\pi) = d_{BV}( \OO^A_{\psi_0} ) \}\\
    &=\{ \pi \in \Pi(G_n)_{\lambda} \ | \ d_{BV}( \OO_{\phi_{\widehat{\pi}}} ) = d_{BV}( \OO_{\phi^0} )  \}\\
    &=\{ \pi \in \Pi(G_n)_{\lambda} \ | \ \phi_{\widehat{\pi}} \in (d_{BV})^{-1}_{\Phi(G_n)_{\lambda}} ( d_{BV}( \OO_{\phi^0} ) )  \},
    \end{align*}
    where $\phi^{0}= \phi_{\widehat{\psi_0}}$, the unique tempered $L$-parameter of $\Phi(G)_{\lambda_{\psi_0}}$. 
\end{remark}

Let us recall the statement of the Conjecture for weak local Arthur packets.

\begin{conj}[{\cite[Conjecture 3.1.2]{CMO22}}]\label{conj weak packet}
Let $\psi$ be a basic local Arthur parameter of $G$ and denote $\lambda$ the (real) infinitesimal parameter associated with $\psi$. Then $\Pi_{\psi}^{\textrm{Weak}}$ is a union of local Arthur packets.
\end{conj}

With the results in Section \ref{sec fiber L-par}, we prove Conjecture \ref{conj weak packet} for the split groups $G_n=\SO_{2n+1}(F)$, $\Sp_{2n}(F)$ or $\SO_{2n}(F)$ assuming Conjecture \ref{conj Jiang}(i) holds for all $\psi$ whose restriction to $W_F$ is trivial.

\begin{thm}\label{thm main}
Let $G_n$ be the split group $\SO_{2n+1}(F)$, $\Sp_{2n}(F)$ or $\SO_{2n}(F)$. Assume that the residue field of $F$ has sufficiently large characteristic.
\begin{enumerate}
    \item [(a)] For any basic local Arthur parameter $\psi_0$ of $G_n$, the weak local Arthur packet $\Pi_{\psi_0}^{\textrm{Weak}}$ is contained in a union of local Arthur packets
        \begin{align}\label{eq weak packet union}
        \Pi_{\psi_0}^{\textrm{Weak}}\subseteq \bigcup_{ \psi \in (d_{BV})_{\Psi(G_n)_{\lambda}}^{-1}(\OO')} \Pi_{\psi},
    \end{align}
     where $\lambda= \lambda_{\psi^0}$, $\OO'=d_{BV}(\OO_{\psi_0}^A)$, and 
     \[(d_{BV})^{-1}_{\Psi(G_n)_{\lambda}}(\OO'):=\{ \psi \in \Psi(G_n)_{\lambda} \ | \ d_{BV}(\OO_{\psi}^A)=\OO' \}.\]
     \item [(b)] Moreover, assume that Conjecture \ref{conj Jiang}(i) holds for all $\psi \in \Psi(G_n)$ whose restriction to $W_F$ is trivial. Then we have the other direction of containment
     \begin{align}\label{eq weak packet union reverse}
        \Pi_{\psi_0}^{\textrm{Weak}}\supseteq \bigcup_{ \psi \in (d_{BV})_{\Psi(G_n)_{\lambda}}^{-1}(\OO')} \Pi_{\psi},
    \end{align}
    which proves Conjecture \ref{conj weak packet}. In particular, by Theorem \ref{thm Jiang B C}, Conjecture \ref{conj weak packet} holds for split $\SO_{2n+1}(F)$ and $\Sp_{2n}(F)$ without the assumption of Conjecture \ref{conj Jiang}(i). 
\end{enumerate}
\end{thm}

\begin{proof}
Note that
\[ \OO'=d_{BV}(\OO_{\psi_0}^A)= d_{BV}(\OO_{\widehat{\psi}_0}^{D})=d_{BV}(\OO_{\phi^0}),\]
where $\phi^0= \phi_{\widehat{\psi}_0}$ is the unique tempered $L$-parameter in $\Phi(G_n)_{\lambda}$. For Part (a), suppose $\pi \in \Pi_{\psi_0}^{\textrm{Weak}}$. Then Theorem \ref{thm computation of WF} gives
\[ \OO'=\WF(\pi)= d_{BV}( \OO_{\phi_{\widehat{\pi}}}). \]
Namely, $\phi_{\widehat{\pi}} \in (d_{BV})_{\Phi(G_n)_{\lambda}}^{-1}(   d_{BV}(\OO_{\phi^0}))$. Therefore, Proposition \ref{prop fiber tempered nilpotent orbit} implies that the $L$-parameter $\phi_{\widehat{\pi}}$ is of Arthur type. Say $\phi_{\widehat{\pi}}=\phi_{\widehat{\psi}}$. Then $\psi \in (d_{BV})_{\Psi(G_n)_{\lambda}}^{-1}(\OO') $. Also, $\widehat{\pi} \in \Pi_{\phi_{\widehat{\pi}}} =\Pi_{\phi_{\widehat{\psi}}}\subseteq \Pi_{\widehat{\psi}}$, and hence $\pi \in \Pi_{\psi}$. This verifies \eqref{eq weak packet union}.

For Part (b), suppose $\pi\in \Pi_{\psi}$ where $\psi \in (d_{BV})_{\Psi(G_n)_{\lambda}}^{-1}(\OO')$. Conjecture \ref{conj Jiang}(i) implies
    \[ \WF(\pi) \leq d_{BV}( \OO_{\psi}^A )=\OO'=d_{BV}( \OO_{\psi_0}^A ).\]
Therefore, $\pi$ is in the weak local Arthur packet $\Pi_{\psi_0}^{\mathrm{Weak}}$. This proves \eqref{eq weak packet union reverse} and completes the proof of the theorem.    
\end{proof}

\begin{remark}
    \begin{enumerate}
    \item [1.] The same proof works for inner forms of split groups $\SO_{2n+1}(F)$, $\Sp_{2n}(F)$ or $\SO_{2n}(F)$ once Arthur's theory on the local Arthur packets is developed. Especially, we need the results that $\Pi_{\phi_{\psi}} \subseteq \Pi_{\psi}$ and $\Pi_{\widehat{\psi}}= \{\widehat{\pi} \ | \ \pi \in \Pi_{\psi}\}$ in the proof.
        \item [2.] If one can verify an analogue of Proposition \ref{prop fiber tempered nilpotent orbit} for any connected reductive algebraic group $G$, inner to split, then Theorem \ref{thm main} also holds for $G$ by similar arguments. However, Conjecture \ref{conj weak packet} may not imply Proposition \ref{prop fiber tempered nilpotent orbit} in general. Therefore, Proposition \ref{prop fiber tempered nilpotent orbit} for $G_n$ has its own interests.
        
        \item [3.]Part (a) of above theorem (together with Arthur's theory) implies that any representation in a weak local Arthur packet is unitary. This proves \cite[Conjecture 3.1.3]{CMO22} for the split group $\SO_{2n+1}(F)$, $\Sp_{2n}(F)$ or $\SO_{2n}(F)$ without assumptions. We record it below using their notation.
    \end{enumerate}
\end{remark}

\begin{thm}\label{thm CMO unitary}
Assume that the residue field of $F$ has sufficiently large characteristic.
     Let $(q^{\half{1}h^{\vee}}, x ,\rho)$ be a Deligne-Langlands-Lusztig parameter of the split groups $\SO_{2n+1}(F)$, $\Sp_{2n}(F)$ or $\SO_{2n}(F)$ such that $x$ belongs to the special piece of $\mathbb{O}^{\vee}$. Then the irreducible representation $X(q^{\half{1}h^{\vee}}, x ,\rho)$ is unitary.
\end{thm}

Finally, we give a remark of a characterization of anti-tempered local Arthur packets, which directly comes from \cite[Theorem 3.0.3]{CMO22}. 

\begin{remark}
    For any basic local Arthur parameter $\psi_0$ of $G$, it is stated in \cite[Theorem 3.0.3]{CMO22} that
\[ \Pi_{\psi_0}= \{ \pi \in \Pi(G)_{\lambda}\ | \   {}^K \mathrm{WF}(\pi)= d_A( \OO_{\psi_0}^A, 1 ) \},\]
 where ${}^K \mathrm{WF}(\pi)$ is the \emph{canonical unramified wavefront set} of $\pi$, $d_A$ is the duality defined in \cite{Ach03}. See \cite{CMO22} for details of these notation. By \cite[Theorem 2.6.2(1)]{CMO22}, we may rewrite it as
    \begin{align}\label{eq characterization of basic packets}
        \Pi_{\psi_0}= \{ \pi \in \Pi(G)_{\lambda}\ | \ d_A(\OO_{\phi_{\widehat{\pi}}},1) = d_A( \OO_{\psi_0}^A, 1 ) \}.
    \end{align}
   The right hand side of \eqref{eq characterization of basic packets} makes sense for any local Arthur parameter $\psi_0$, which is not necessarily basic. Indeed, the same proof of \cite[Theorem 3.0.3]{CMO22} implies that \eqref{eq characterization of basic packets} holds for any anti-tempered local Arthur parameter $\psi_0$ of $G$. This gives a characterization of anti-tempered local Arthur packets.
\end{remark} 

Similarly, if we replace the $\WF(\pi)$ by $d_{BV}(\OO_{\phi_{\widehat{\pi}}})$ in the definition of weak local Arthur packet \eqref{eq weak packet leq}, then we can generalize Theorem \ref{thm main} to any anti-tempered local Arthur parameter, which is not necessarily basic. More precisely, we have the following theorem. Note that in this case we do not need the characteristic assumption of the residue field of $F$ since we don't need to make use of Theorem \ref{thm computation of WF}.

\begin{thm}\label{thm main 2}
Let $G_n$ be the split group $\SO_{2n+1}(F)$, $\Sp_{2n}(F)$ or $\SO_{2n}(F)$. For any anti-tempered local Arthur parameter $\psi_0$, we denote $\lambda:=\lambda_{\psi_0}$ and $\OO':= d_{BV}(\OO_{\psi_0}^{A})$. Consider the set of representations
\[ \Pi_{\psi_0}^{\textrm{Weak}}:= \{ \pi \in \Pi(G_n)_{\lambda} \ | \ d_{BV}(\OO_{\phi_{\widehat{\pi}}}) \leq d_{BV}(\OO_{\psi_0}^{A})\}.\]
We have an inclusion
\begin{align}\label{eq main 2}
    \Pi_{\psi_0}^{\textrm{Weak}} \subseteq \bigcup_{ \psi \in (d_{BV})_{\Psi(G_n)_{\lambda}}^{-1}(\OO')} \Pi_{\psi}.
\end{align}
Moreover, the above inclusion is an equality if $\OO_{\phi_{\pi}} \geq \OO_{\phi_{\psi}}$ for any $\psi \in \Psi(G_n)_{\lambda}$ and $\pi \in \Pi_{\psi}$, which is already verified for the split groups $\SO_{2n+1}(F)$ and $\Sp_{2n}(F)$ in \cite[Theorem 1.15, Corollary 4.12(2)]{HLLZ22}.
\end{thm}

\begin{proof}
    The proof of the inclusion \eqref{eq main 2} is exactly the same as the proof of Theorem \ref{thm main}(a) (without using Theorem \ref{thm computation of WF}), which we omit. Conversely, suppose $\pi$ is in the right hand side of \eqref{eq main 2}, i.e., there exists a local Arthur parameter $\psi$ such that $\pi \in \Pi_{\psi}$ and $d_{BV}(\OO_{\psi}^{A})= d_{BV}(\OO_{\psi_0}^A)$. Then $\widehat{\pi} \in \Pi_{\widehat{\psi}}$ and the assumption gives $\OO_{\phi_{\widehat{\pi}}} \geq \OO_{\phi_{\widehat{\psi}}}$. Taking Barbasch-Vogan duality, we obtain 
    \[ d_{BV}(\OO_{\phi_{\widehat{\pi}}}) \leq d_{BV}(\OO_{\phi_{\widehat{\psi}}})= d_{BV}(\OO_{\psi_0}^{A}),\]
    which implies that $\pi \in \Pi_{\psi_0}^{\textrm{Weak}}$. This completes the proof of the theorem.
\end{proof}

\section{Generalizations of weak local Arthur packets and examples}\label{sec: generalization}
In this section, we discuss possible generalization for the definition of weak local Arthur packets such that Conjecture \ref{conj weak packet intro} or \eqref{eq weak packet union} have a chance to be true. Throughout the section, we let $G= \RG(F)$ be the $F$-point of a connected reductive algebraic group $\RG$ defined over $F$, and assume there is a local Arthur packets theory for $G$ as conjectured in \cite[Conjecture 6.1]{Art89}. Let $\lambda$ be any infinitesimal parameter of $G$. Without loss of generality, we may assume that $\lambda =\lambda_{\psi}$ for some local Arthur parameter $\psi$, otherwise there is no representation of Arthur type in $\Pi(G)_{\lambda}$.

A first naive generalization is to consider the set
\begin{align}\label{eq generalization leq}
  \Pi_{\OO',\lambda}^{\textrm{WF}}:=  \{ \pi \in \Pi(G)_{\lambda} \ | \ \WF(\pi) \leq \OO'  \},
\end{align}
for any nilpotent orbit $\OO'$ of $\RG(\BC)$. However, $\Pi_{\OO',\lambda}^{\textrm{WF}}$ is not always a union of local Arthur packets since it may contain representations not of Arthur type, see the following example.  

\begin{exmp}\label{sec5 exam 1}
Take $\phi^0$ to be a tempered unramified $L$-parameter of $\SO_{2n+1}(F)$ or $\Sp_{2n}(F)$ such that $\lambda:=\lambda_{\phi^0}$ is real, and that there exists a representation $\pi_{bad}\in \Pi(G_n)_{\lambda}$ that is not of Arthur type. For any $\pi \in \Pi(G_n)_{\lambda}$, we have
\[ \WF(\pi)= \{ d_{BV}( \OO_{\phi_{\widehat{\pi}}})\}\]
by Theorem \ref{thm computation of WF}. Let $\phi_0$ be the closed $L$-parameter in $\Phi(G_n)_{\lambda}$. Then since $\OO_{\phi_{\widehat{\pi}}} \geq \OO_{\phi_0}$, we conclude that
\[ \WF(\pi) \leq d_{BV}(\OO_{\phi_0}).\]
Therefore, taking $\OO'=d_{BV}(\OO_{\phi_0})$, we have 
\[\Pi_{\OO',\lambda}^{\textrm{WF}}= \Pi(G_n)_{\lambda},\]
 which contains $\pi_{bad}$ not of Arthur type.
\end{exmp}

A natural modification of \eqref{eq generalization leq} is to add the condition of Arthur type.
\begin{align}\label{eq generalization leq A type}
    \Pi_{\OO',\lambda}^{\textrm{WF,A}} :=  \{ \pi \in \Pi(G)_{\lambda} \text{ of Arthur type} \ | \ \WF(\pi) \leq \OO'  \}.
\end{align}
Then assuming Conjecture \ref{conj Jiang}(i), we have
\begin{align}\label{eq containment generalization leq A type}
     \Pi_{\OO',\lambda}^{\textrm{WF,A}} \supseteq \bigcup_{ \substack{\psi \in \Psi(G_n)_{\lambda},\  d_{BV}(\OO_{\psi}^A) \leq \OO'} } \Pi_{\psi}.
\end{align}
Moreover, assuming Conjecture \ref{conj Jiang}(ii), the right hand side of \eqref{eq containment generalization leq A type} is exactly the union of all local Arthur packets contained in $ \Pi_{\OO',\lambda}^{\textrm{WF,A}}$, see the proof of Proposition \ref{prop generalization} below for details. However, the containment \eqref{eq containment generalization leq A type} can be strict even if we assume $\OO'= d_{BV}(\OO_{\psi}^{A})$ for some local Arthur parameter $\psi$ in $\Psi(G)_{\lambda}$, as shown in the following example.

\begin{exmp}\label{sec5 exam 2}
In this example, we adopt the notation in \cite{AM20} for enhanced $L$-parameter, which is called $L$-data there. Consider the following unipotent representations of $\SO_{11}(F)$ with real infinitesimal parameter $\lambda$,
\begin{align*}
    \pi&= L(\Delta[-1/2, -3/2]; \pi((1/2)^{-}, (3/2)^{-})),\\
    \widehat{\pi}&=L( \Delta[-3/2,-3/2]; \pi((1/2)^{-},(1/2)^{-},(3/2)^{+}) ).
\end{align*}
The $L$-parameter associated with $\widehat{\pi}$ is 
\[\phi_{\widehat{\pi}}= |\cdot|^{-3/2}\otimes S_1+ |\cdot|^{3/2}\otimes S_1+ 1\otimes S_2+1 \otimes S_2+ 1 \otimes S_4,\]
which is not of Arthur type. We have (identifying nilpotent orbits of $\SO_{11}(\BC)$ as partitions)
\[ \WF(\pi)= \{d_{BV}( \OO_{\phi_{\widehat{\pi}} })\}= \{d_{BV}([4,2,2,1,1])\}= \{[5,3,1^3]\}.\]

On the other hand, the representation $\pi$ is of Arthur type, and 
\[ \Psi(\pi):= \{\psi \ | \ \pi \in \Pi_{\psi}\}= \{\psi_1, \psi_2\},\]
where 
\begin{align*}
    \psi_1&= 1\otimes S_2 \otimes S_3 + 1 \otimes S_4 \otimes S_1,\\
    \psi_2&= 1 \otimes S_1 \otimes S_2+1 \otimes S_1 \otimes S_4+1 \otimes S_4 \otimes S_1.
\end{align*}
One can compute that
\begin{align*}
    d_{BV}(\OO^{A}_{\psi_1})= [7,2^2],\ \ 
        d_{BV}(\OO^{A}_{\psi_2})= [7,1^4].
\end{align*}
By Conjecture \ref{conj Jiang} (ii), for $i=1,2$, there exists at least one representation $\pi_i \in \Pi_{\psi_i}$ with $\WF(\pi_i)= d_{BV}(\OO^{A}_{\psi_i})$. Indeed, by the explicit formula in Theorem \ref{thm computation of WF}, we can take 
\begin{align*}
    \pi_1&=L( \Delta[-1/2,-3/2]; \pi( (1/2)^+, (3/2)^+ ),\\
    \pi_2&=L( \Delta[-3/2,-3/2], \Delta[ -1/2,-1/2]; \pi( (1/2)^+, (3/2)^+ ).
\end{align*}
Therefore, any local Arthur packet $\Pi_{\psi}$ that contains $\pi$ is not contained in the set
\[
\Pi_{[5,3,1^3],\lambda}^{\textrm{WF,A}}=\{ \pi' \in \Pi(\SO_{11}(F))_{\lambda} \ | \ \WF(\pi') \leq [5,3,1^3],\text{ and } \pi' \text{ is of Arthur type. }  \}.  \]
However, $\pi$ is in $\Pi_{[5,3,1^3],\lambda}^{\textrm{WF,A}}$. Therefore, $\Pi_{[5,3,1^3],\lambda}^{\textrm{WF,A}}$ can not be written as a union of local Arthur packets. We remark that $[5,3,1^3]= d_{BV}( \OO^{A}_{\psi})$ where  $\psi= 1 \otimes S_{3}\otimes S_{2}+1 \otimes S_1 \otimes S_4 \in \Psi(\SO_{11}(F))_{\lambda}$.
\end{exmp}

The failure of the containment \eqref{eq containment generalization leq A type} in the above example comes from Conjecture \ref{conj Jiang}(ii) and the fact that there does not exist a $\psi$ in  
\[\Psi(\pi):= \{ \psi \in \Psi(G)\ | \ \pi \in \Pi_{\psi}\}\]
such that $d_{BV}(\OO_{\psi}^{A}) \leq \OO'.$ This suggests the final modification for the definition of weak local Arthur packets. We summarize these phenomena in the following proposition.

\begin{prop}\label{prop generalization}
Let $\mathrm{G}$ be a connected reductive group and $G=\mathrm{G}(F)$. 
Assume that there is a local Arthur packets theory for $G$ as conjectured in \cite[Conjecture 6.1]{Art89}.
    Assume Conjecture \ref{conj Jiang} holds for the group $G=\RG(F)$. For any nilpotent orbit $\OO'$ of $\RG(\BC)$ and any infinitesimal parameter $\lambda$ of $G$, we define the weak local Arthur packet as follows
    \begin{align}\label{eq generalization final}
        \Pi_{\OO',\lambda}^{\textrm{Weak}}:=\{\pi \in \Pi(G)_{\lambda} \textrm{ of Arthur type}\ | \ \exists \psi \in \Psi(\pi) \text{ such that }d_{BV}(\OO_{\psi}^{A}) \leq \OO'\}.
    \end{align}
    Then 
    \[\Pi_{\OO',\lambda}^{\textrm{Weak}}=   \bigcup_{\psi \in \Psi(G)_{\lambda},\ \Pi_{\psi}\subseteq \Pi_{\OO',\lambda}^{\textrm{WF,A}} } \Pi_{\psi} \ \subseteq\  \Pi_{\OO',\lambda}^{\textrm{WF,A}}, \]
    where the containment can be strict.
\end{prop}

\begin{proof}
   Suppose $\pi \in \Pi_{\OO',\lambda}^{\textrm{Weak}}$. Then there exists a $\psi$ such that $ d_{BV}(\OO_{\psi}^{A}) \leq \OO'$ and $\pi \in \Pi_{\psi}$. By Conjecture \ref{conj Jiang}(i), for any $\pi' \in \Pi_{\psi}$, we have
   \[ \WF(\pi') \leq d_{BV}(\OO_{\psi}^{A}) \leq \OO',\]
   and hence $\Pi_{\psi}\subseteq \Pi_{\OO',\lambda}^{\textrm{WF,A}} $. 
   
   Conversely, suppose $\pi \in \Pi_{\psi'}$ where $\Pi_{\psi'} \subseteq \Pi_{\OO',\lambda}^{\textrm{WF,A}}$. By Conjecture \ref{conj Jiang}(ii), there exists a $\pi'\in \Pi_{\psi'}$ such that $\WF(\pi')= \{ d_{BV}(\OO_{\psi'}^A)\}$, and hence $ d_{BV}(\OO_{\psi'}^A) \leq \OO'$.
   This completes the proof of the proposition.
 \end{proof}

Assuming Conjecture 
\ref{conj Jiang}, one can see from Proposition \ref{prop generalization}
that 
$\Pi_{\OO',\lambda}^{\textrm{Weak}}$ is 
the maximal subset of $\Pi(G)_{\lambda}$ with the following properties.
\begin{enumerate}
    \item [$\oldbullet$]  $\Pi_{\OO',\lambda}^{\textrm{Weak}} \subseteq \{ \pi \in \Pi(G)_{\lambda}\ | \ \WF(\pi) \leq \OO' \}$.
    \item [$\oldbullet$] $\Pi_{\OO',\lambda}^{\textrm{Weak}}$ is a union of local Arthur packets.
\end{enumerate} 
Hence, if Conjecture \ref{conj weak packet} holds for a basic local Arthur parameter $\psi_0$ of $G$, then
\[ \Pi_{\psi_0}^{\textrm{Weak}}=\Pi_{\OO',\lambda}^{\textrm{Weak}},\]
where $\OO'= d_{BV}(\OO_{\psi_0}^A)$ and $\lambda=\lambda_{\psi_0}$. 
Therefore, the set $\Pi_{\OO',\lambda}^{\textrm{Weak}}$ can be regarded as a natural generalization of $\Pi_{\psi_0}^{\textrm{Weak}}$ and reveals the implications of Conjecture \ref{conj Jiang}.

\end{document}